\documentclass[english,12pt]{article}
\usepackage[english]{babel}
\usepackage[normalem]{ulem}
\usepackage[hmargin=2.3cm,vmargin=2.36cm]{geometry}
\usepackage{amssymb,amsmath,amsthm,mathtools,braket,aligned-overset}
\usepackage{thmtools,thm-restate}

\usepackage{graphicx,xcolor}
\usepackage{float}
\usepackage{caption,subcaption}
\usepackage{tikz}
\usetikzlibrary{calc,decorations.pathmorphing,decorations.text,decorations.markings,matrix,shadings,shapes.geometric}

\usepackage{array,booktabs}
\usepackage{enumerate,enumitem}

\usepackage{algorithm}
\usepackage{algpseudocode}

\usepackage{appendix}
\usepackage{titlesec}
\usepackage{authblk}
\usepackage{bm}

\usepackage{hyperref}
\usepackage[capitalise,nameinlink,noabbrev]{cleveref}

\usepackage[square,sort,comma,numbers]{natbib}
\allowdisplaybreaks

\theoremstyle{plain}
\newtheorem{theorem}{Theorem}[section]
\newtheorem{question}[theorem]{Question}

\newtheorem{claim}{Claim}
\newtheorem{lemma}[theorem]{Lemma}

\newtheorem{proposition}[theorem]{Proposition}

\newtheorem{definition}[theorem]{Definition}
\newtheorem{remark}[theorem]{Remark}

\newlength{\bibitemsep}
\newlength{\bibparskip}
\let\oldthebibliography\thebibliography
\renewcommand\thebibliography[1]{
  \oldthebibliography{#1}
  \setlength{\parskip}{\bibitemsep}
  \setlength{\itemsep}{\bibparskip}
}

\newenvironment{proof*}[1][Proof]
{\begin{proof}[#1]}
{\end{proof}}

\newcommand{\diag}{\operatorname{diag}}
\newcommand{\one}{\mathbf 1}

\newcommand{\Z}{\mathbb Z}

\title{On the structure of graphs with given odd girth \\ and large algebraic connectivity}

\author[1]{Zhengbo Chen} 
\author[2]{Chenxing Li}
\author[2]{Zhouningxin Wang}

\affil[1]{\small School of Mathematical Sciences, Shanghai Jiao Tong University, Shanghai 200240, China.}
\affil[2]{\small School of Mathematical Sciences and LPMC, Nankai University, Tianjin 300071, China. \linebreak Emails: czb911@sjtu.edu.cn,  chenxingli@mail.nankai.edu.cn, wangzhou@nankai.edu.cn.}

\date{}

\begin{document}
\maketitle

\begin{abstract}
A classical result of Andr\'asfai, Erd\H{o}s, and S\'os states that every $n$-vertex graph with odd girth at least $2k+1$ and minimum
degree larger than $\frac{2n}{2k+1}$
is bipartite. Rather than imposing a minimum-degree condition, in this paper we investigate conditions on algebraic connectivity that force
graphs of given odd girth to have a simple structure. The algebraic connectivity of a graph $G$, denoted by $\mu_2(G)$, is the second smallest eigenvalue of its Laplacian matrix. Our main results are as follows.
\begin{itemize}
    \item Every $n$-vertex triangle-free graph $G$ with $\mu_2(G)\geq \frac{n}{3}$ is bipartite. Moreover, the constant $\frac{1}{3}$ is asymptotically best possible.
    \item For $k\geq 3$, every $n$-vertex graph $G$ of odd girth at least $2k+1$ with $\mu_2(G)>\frac{4n}{6k-1}$ is bipartite.
    \item For $k\geq 22$, every $n$-vertex graph $G$ of odd girth at least $2k+1$ with $\mu_2(G)>\frac{3456n}{k^3}$ is bipartite. Moreover, the term $k^{-3}$ is asymptotically best possible.
\end{itemize}
\end{abstract}


\section{Introduction}
The study of minimum-degree conditions that force colorability or homomorphisms to prescribed targets in graphs of large odd girth has long been a central topic in extremal graph theory. The \emph{odd girth} of a graph $G$ is the length of a shortest odd cycle in $G$. In particular, a graph is \emph{triangle-free} if its odd girth is at least $5$. A fundamental result of Andr\'asfai, Erd\H{o}s, and S\'os~\cite{AES1974} characterized the extremal structures of $n$-vertex graphs of odd girth at least $2k+1$ with $\delta(G)>\frac{2n}{2k+1}$. The bound $\frac{2n}{2k+1}$ is tight, as whenever $2k+1$ divides $n$, any $n$-vertex balanced blow-up of $C_{2k+1}$ is nonbipartite, has odd girth $2k+1$, and is $\frac{2n}{2k+1}$-regular.

\begin{theorem}[Andr\'asfai, Erd\H{o}s, and S\'os~\cite{AES1974}]\label{thm:triangle-free-minimum-degree}
Let $k$ be a positive integer with $k\geq 2$. Every $n$-vertex graph of odd girth at least $2k+1$ with $\delta(G)>\frac{2n}{2k+1}$ is bipartite. 
\end{theorem}

A \emph{homomorphism} of a graph $G$ to another graph $H$ is a vertex mapping $\varphi:V(G)\to V(H)$ such that $uv\in E(G)$ implies $\varphi(u)\varphi(v)\in E(H)$. A graph is bipartite if and only if it admits a homomorphism to $K_2$. 

Beyond bipartiteness, homomorphisms to odd cycles have been extensively studied. H\"aggkvist~\cite{H1982} proved that every $n$-vertex triangle-free graph $G$ with $\delta(G)>\frac{3n}{8}$ admits a homomorphism to $C_5$, while H\"aggkvist and Jin~\cite{HJ1998} showed that every $n$-vertex graph $G$ of odd girth at least seven with $\delta(G)>\frac{n}{4}$ admits a homomorphism to $C_7$. Messuti and Schacht~\cite{MS2015} subsequently extended these results by proving that the conditions $\delta(G)>\frac{3n}{4k}$ and odd girth at least $2k+1$ can guarantee a homomorphism to the odd cycle $C_{2k+1}$. This bound $\frac{3n}{4k}$ is tight, as witnessed by balanced blow-ups of the circular clique $K_{4k/(2k-1)}$ (defined later in~\Cref{def:Circular clique}).

\begin{theorem}[Messuti and Schacht~\cite{MS2015}]\label{thm:large-girth-minimum-degree}
Every $n$-vertex graph $G$ of odd girth at least $2k+1$ with $\delta(G)>\frac{3n}{4k}$ admits a homomorphism to $C_{2k+1}$.
\end{theorem}

For other homomorphism targets, various results were obtained. Jin~\cite{J1993} showed that, for each $1\leq k\leq 9$, every $n$-vertex triangle-free graph $G$
satisfying $\delta(G)>\lfloor\frac{(k+1)n}{3k+2}\rfloor$ admits a homomorphism to the circular clique $K_{(3k-1)/k}$, and the corresponding minimum-degree bound is sharp. The Brandt--Thomass\'e structure theorem~\cite{BT2005,LPR2021} asserted that every $n$-vertex maximal triangle-free graph $G$ with $\delta(G)>\frac{n}{3}$ is a complete blow-up of an Andr\'asfai graph or a Vega graph. \L{}uczak~\cite{L2006} showed that, for any given $\varepsilon>0$, every $n$-vertex triangle-free graph $G$ with $\delta(G)\geq (\frac{1}{3}+\varepsilon) n$ admits a homomorphism to a triangle-free graph $H(\varepsilon)$ of bounded order. 

For odd girth $7$, Brandt and Ribe-Baumann~\cite{BR2009} showed that a weaker condition $\delta(G)>\frac{4n}{17}$ forces a homomorphism of $G$ to $K_{12/5}$. Recently, Lu and Wang~\cite{LW2026+} extended this phenomenon by proving that, for every $k\geq2$, every $n$-vertex graph $G$ of odd girth at least $2k+1$ and with $\delta(G)>\frac{4n}{6k-1}$ admits a homomorphism to $K_{4k/(2k-1)}$. Ebsen and Schacht~\cite{ES2020} showed that, for every $k\geq3$ and every $\varepsilon>0$, every $n$-vertex graph $G$ of odd girth at least $2k+1$ with $\delta(G)\geq (\frac{1}{2k-1}+\varepsilon)n$ admits a homomorphism to a graph $H(k,\varepsilon)$ of bounded order that has odd girth at least $2k+1$.

\medskip
In this paper, we study a Laplacian-spectral counterpart of this theory. For any $n$-vertex graph $G$ with $V(G)=\{v_1, \ldots, v_n\}$, let $D(G)$ denote the $n\times n$ diagonal matrix $(d_{ij})_{n\times n}$ defined by $d_{ij}=d_G(v_i)$ (where $d_G(v_i)$ denotes the degree of $v_i$ in $G$) if $i=j$ and $d_{ij}=0$ otherwise. For an $n$-vertex graph $G$ and its adjacency matrix $A(G)$, let $L(G):=D(G)-A(G)$
be its \emph{Laplacian matrix}, and let $\mu_1(G), \ldots, \mu_n(G)$ denote the eigenvalues of $L(G)$ (that is, $\operatorname{Spec}(L(G))=\{\mu_1(G), \ldots, \mu_n(G)\}$), called \emph{Laplacian eigenvalues} of $G$, with $\mu_1(G)\leq \mu_2(G)\le\cdots\le \mu_n(G)$.
Note that $\mu_1(G)=0$. The parameter $\mu_2(G)$ is called the
\emph{algebraic connectivity} of $G$.  Fiedler's classical result~\cite{F1973} shows that every non-complete graph $G$ satisfies $\mu_2(G)\leq \delta(G)$.
The Andr\'asfai--Erd\H{o}s--S\'os theorem (\Cref{thm:triangle-free-minimum-degree}), together with this relation, gives a trivial condition $\mu_2(G)>\frac{2n}{2k+1}$ for forcing a graph of odd girth at least $2k+1$ to be bipartite. For analogues of the Andr\'asfai--Erd\H{o}s--S\'os theorem, we establish better bounds on the algebraic connectivity conditions. Our main results are the following. 

\begin{theorem}\label{thm:main}
Every $n$-vertex triangle-free graph $G$ with $\mu_2(G)\geq \frac{n}{3}$ is bipartite.
Moreover, the constant $\frac{1}{3}$ is asymptotically best possible.
\end{theorem}

\begin{theorem}\label{thm:main-girth}
Let $k$ be an integer with $k\geq 3$. Every $n$-vertex graph $G$ of odd girth at least $2k+1$ with $\mu_2(G)>\frac{4n}{6k-1}$ is bipartite.
\end{theorem}

However, the constant $\frac{4}{6k-1}$ does not appear to be optimal. For sufficiently large $k$, we have the following estimate.

\begin{theorem}\label{thm:main-girth_asymptotical}
For any integer $k$ with $k\geq 22$, every $n$-vertex graph $G$ of odd girth at least $2k+1$ with $\mu_2(G)>\frac{3456n}{k^3}$ is bipartite. Moreover, the order $k^{-3}$ is asymptotically best possible.
\end{theorem}

We also investigate a spectral analogue of the Messuti--Schacht theorem (\Cref{thm:large-girth-minimum-degree}), seeking an algebraic-connectivity condition sharper than the immediate bound $\mu_2(G)>\frac{3n}{4k}$ that forces a graph of odd girth at least $2k+1$ to admit a homomorphism to $C_{2k+1}$. At present, however, our methods do not yield any improvement over the condition obtained by
simply forcing the graph to be bipartite. We therefore collect some remarks and open questions in~\Cref{sec:Remarks} for future investigation.

\section{Preliminaries}
All graphs are finite and simple. We first introduce some families of graphs, which play important roles in our structural results.

\begin{definition}{\rm \cite{HN2004}}
\label{def:Circular clique}
For positive integers $p$ and $q$ with $p\ge 2q$, the
\emph{circular clique} $K_{p/q}$ is a graph with the vertex set
$\mathbb{Z}_p=\{0,1,\ldots,p-1\}$ in which two distinct vertices
$i,j\in\mathbb{Z}_p$ are adjacent if $q\le |i-j|\le p-q.$
\end{definition}

We first focus on circular cliques $K_{4k/(2k-1)}$ and $K_{(2k+1)/k}$ (which is isomorphic to $C_{2k+1}$); for those two families, we have the following facts.

\begin{lemma}[Lemma~6.6~\cite{HN2004}]\label{lem:K_4k/2k-1ToC_2k+1}
Let $k$ be a positive integer. 
\begin{itemize}
\setlength{\itemsep}{0em}
    \item $K_{4k/(2k-1)}$ admits no homomorphism to $C_{2k+1}$.
    \item For any vertex $v$ of $K_{4k/(2k-1)}$, $K_{4k/(2k-1)}-v$ admits a surjective homomorphism to $C_{2k+1}$.
\end{itemize}
\end{lemma}

We focus on another special circular clique $K_{(3k-1)/k}$, which is well-known as \emph{the Andr\'asfai graph} $\Gamma_k$ \cite{A1962, LPR2021}. In particular, $\Gamma_1\equiv K_2$ and $\Gamma_2\equiv C_5$. For $k\geq 2$, the Andr\'asfai graph $\Gamma_k$ is not bipartite. Moreover, the \emph{Vega graphs} form another important family.

\begin{definition}{\rm \cite{A1962, LPR2021}}\label{def:vega}
Let $i\geq 2$ be an integer, and let $\mu,\nu\in\{0,1\}$. The \emph{Vega graph} $\Upsilon_i^{\mu\nu}$ is constructed as follows: Start with the Andr\'asfai graph $\Gamma_i$, where
$V(\Gamma_i)=\{v_0,\ldots,v_{3i-2}\},$
and add eight new vertices $a,b,c,u,v,w,x,y$. Add the edge $xy$, and add edges so that the vertices $a,v,c,u,b,w$, in this cyclic order, induce a $6$-cycle. Join $x$ to $a,b,c$ and $y$ to $u,v,w$. Moreover, join each of $a$ and $u$ to every vertex in $\{v_0,\ldots,v_{i-1}\}$, each of $b$ and $v$ to every vertex in $\{v_i,\ldots,v_{2i-1}\}$, and each of $c$ and $w$ to every vertex in $\{v_{2i},\ldots,v_{3i-2}\}$. Finally, delete $y$ if $\mu=1$, and delete $v_{2i-1}$ if $\nu=1$.
\end{definition}

For example, see~\Cref{fig:Vega211}, where two different labelings of $\Upsilon_2^{11}$ (also known as Gr\"{o}tzsch graph) are shown.

\begin{figure}[H]
\centering
\resizebox{0.85\textwidth}{!}{
\begin{tikzpicture}[
    scale=0.88,
    every node/.style={font=\small},
    vtx/.style={circle,fill=black,inner sep=1.8pt},
    delvtx/.style={circle,draw=black,fill=gray,inner sep=1.8pt},
    six/.style={line width=0.9pt},
    five/.style={line width=0.9pt},
    link/.style={line width=0.9pt},
    deletededge/.style={line width=0.9pt,dashed},
    deletedfive/.style={line width=0.9pt,dashed}
]

\begin{scope}[xshift=-8.3cm]
\coordinate (La) at (90:4.0);
\coordinate (Lv) at (30:4.0);
\coordinate (Lc) at (-30:4.0);
\coordinate (Lu) at (-90:4.0);
\coordinate (Lb) at (-150:4.0);
\coordinate (Lw) at (150:4.0);
\coordinate (Lv3) at (90:1.55);
\coordinate (Lv1) at (162:1.55);
\coordinate (Lv4) at (234:1.55);
\coordinate (Lv2) at (306:1.55);
\coordinate (Lv0) at (18:1.55);
\coordinate (Ly) at (0,5.55);
\coordinate (Lx) at (0,-5.55);

\node[delvtx,label=above:$v_3(\text{deleted})$] at (Lv3) {};
\node[delvtx,label=above:$y(\text{deleted})$] at (Ly) {};

\node[vtx,label=above left:$a$] at (La) {};
\node[vtx,label=below right:$v$] at (Lv) {};
\node[vtx,label=above right:$c$] at (Lc) {};
\node[vtx,label=below:$u$] at (Lu) {};
\node[vtx,label=above left:$b$] at (Lb) {};
\node[vtx,label=below left:$w$] at (Lw) {};
\node[vtx,label=left:$v_1$] at (Lv1) {};
\node[vtx,label=below left:$v_4$] at (Lv4) {};
\node[vtx,label=below right:$v_2$] at (Lv2) {};
\node[vtx,label=right:$v_0$] at (Lv0) {};
\node[vtx,label=below:$x$] at (Lx) {};

\draw[six]
(La)--(Lv)--(Lc)--(Lu)--(Lb)--(Lw)--cycle;

\draw[five]
(Lv0)--(Lv2)--(Lv4)--(Lv1);

\draw[deletedfive]
(Lv1)--(Lv3)--(Lv0);

\draw[link]
(La) .. controls (1.25,3.25) and (1.85,2.20) .. (Lv0);

\draw[link]
(La) .. controls (-1.25,3.25) and (-1.85,2.20) .. (Lv1);

\draw[link]
(Lu) .. controls (2.35,-3.05) and (2.40,0.10) .. (Lv0);

\draw[link]
(Lu) .. controls (-2.35,-3.05) and (-2.40,0.10) .. (Lv1);

\draw[link]
(Lb) .. controls (-2.65,-2.40) and (-0.15,-2.05) .. (Lv2);

\draw[deletededge]
(Lb) .. controls (-3.00,0.20) and (-1.90,2.15) .. (Lv3);

\draw[deletededge]
(Lv) .. controls (2.70,2.55) and (1.25,2.15) .. (Lv3);

\draw[link]
(Lv) .. controls (2.85,0.20) and (2.05,-0.75) .. (Lv2);

\draw[link]
(Lc) .. controls (2.65,-2.40) and (0.15,-2.05) .. (Lv4);

\draw[link]
(Lw) .. controls (-2.95,1.10) and (-2.15,-0.65) .. (Lv4);

\draw[link]
(Lx)
.. controls (1.80,-5.35) and (3.20,-3.85) ..
(Lc);

\draw[link]
(Lx)
.. controls (-1.80,-5.35) and (-3.20,-3.85) ..
(Lb);

\draw[link]
(Lx)
.. controls (6.35,-5.10) and (6.35,4.15) ..
(La);

\draw[deletededge]
(Ly)
.. controls (1.75,5.40) and (3.15,3.75) ..
(Lv);

\draw[deletededge]
(Ly)
.. controls (-1.75,5.40) and (-3.15,3.75) ..
(Lw);

\draw[deletededge]
(Ly)
.. controls (6.25,4.55) and (6.25,-4.10) ..
(Lu);

\draw[deletededge]
(Ly)
.. controls (-6.25,4.70) and (-6.25,-4.70) ..
(Lx);

\end{scope}

\begin{scope}[xshift=4.9cm, scale=1.4, transform shape]

\coordinate (Rp0) at (0,3.25);
\coordinate (Rp1) at (2.85,1.25);
\coordinate (Rp2) at (1.75,-1.95);
\coordinate (Rp3) at (-1.75,-1.95);
\coordinate (Rp4) at (-2.85,1.25);

\coordinate (Rq0) at (0,1.75);
\coordinate (Rq1) at (1.28,0.78);
\coordinate (Rq2) at (0.88,-0.65);
\coordinate (Rq3) at (-0.88,-0.65);
\coordinate (Rq4) at (-1.28,0.78);

\coordinate (Rs) at (0,0.28);

\draw[link] (Rp0)--(Rp1)--(Rp2)--(Rp3)--(Rp4)--(Rp0);

\draw[link] (Rq0)--(Rp4) (Rq0)--(Rp1);
\draw[link] (Rq1)--(Rp0) (Rq1)--(Rp2);
\draw[link] (Rq2)--(Rp1) (Rq2)--(Rp3);
\draw[link] (Rq3)--(Rp2) (Rq3)--(Rp4);
\draw[link] (Rq4)--(Rp3) (Rq4)--(Rp0);

\draw[link] (Rs)--(Rq0);
\draw[link] (Rs)--(Rq1);
\draw[link] (Rs)--(Rq2);
\draw[link] (Rs)--(Rq3);
\draw[link] (Rs)--(Rq4);

\node[vtx,label=above:$p_0$] at (Rp0) {};
\node[vtx,label=right:$p_1$] at (Rp1) {};
\node[vtx,label=below right:$p_2$] at (Rp2) {};
\node[vtx,label=below left:$p_3$] at (Rp3) {};
\node[vtx,label=left:$p_4$] at (Rp4) {};

\node[vtx,label=above:$q_0$] at (Rq0) {};
\node[vtx,label=right:$q_1$] at (Rq1) {};
\node[vtx,label=right:$q_2$] at (Rq2) {};
\node[vtx,label=left:$q_3$] at (Rq3) {};
\node[vtx,label=left:$q_4$] at (Rq4) {};

\node[vtx,label=below:$s$] at (Rs) {};

\end{scope}

\end{tikzpicture}
}

\caption{Two drawings of the graph $\Upsilon_2^{11}$ with $p_0=v_2,p_1=b,p_2=u,p_3=c,p_4=v_4, q_0=w,q_1=v_0,q_2=x,q_3=v_1,q_4=v,$ and $s=a$}
\label{fig:Vega211}
\end{figure}
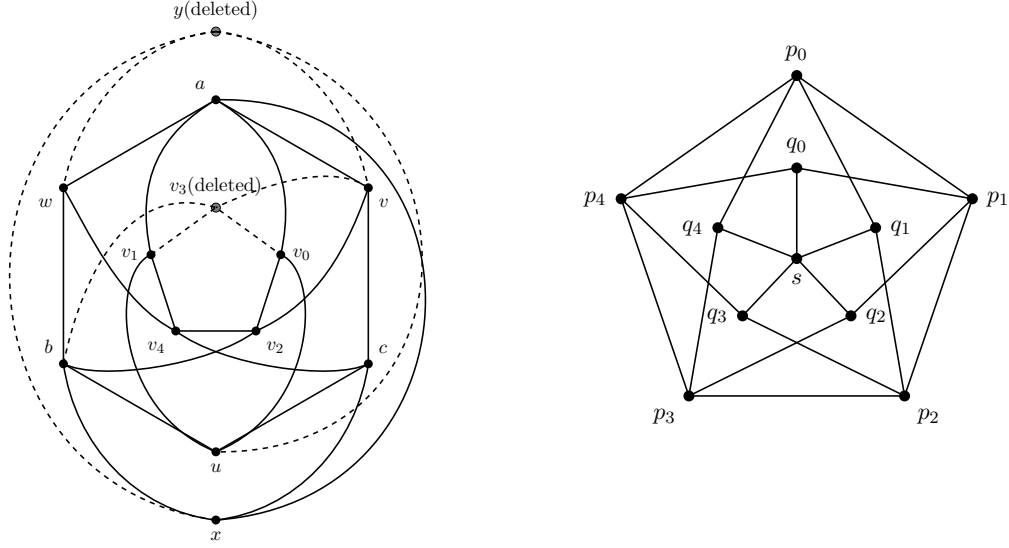

We then provide some well-known facts about the algebraic connectivity of a graph $G$.

\begin{lemma}[Fiedler~\cite{F1973}]\label{lem:Fiedler}
For any non-complete graph $G$, $\mu_2(G)\leq \delta(G)$.
\end{lemma}

\begin{lemma}[Lemma~3.2.1~\cite{BH2012}]
\label{lem:edge-monotone}
For any two graphs $G$ and $G'$ with $V(G)=V(G')$ and $E(G)\subseteq E(G')$, for each $j\in \{1,\ldots, |V(G)|\}$, we have that $\mu_j(G)\le \mu_j(G')$.
\end{lemma}

\begin{lemma}\label{lem:precise-alg-connectivity}
The following claims hold.
\begin{enumerate}[label=(\roman*)]
\setlength{\itemsep}{0em}
    \item\label{lem:Kab} {\rm \cite{F1973}} For any positive integers $s,t$ with $s\ge t$ and $s\geq 2$, $\mu_2(K_{s,t})=t$. 
    \item\label{lem:C_2k+1} {\rm \cite{BH2012}} For each integer $k$ with $k\geq 2$, $\mu_2(C_{2k+1})=2-2\cos\Big(\frac{2\pi}{2k+1}\Big)$.
\end{enumerate}
\end{lemma}

\subsection{The normalized quotient Laplacian matrix of complete blow-ups}

Let $F$ be a graph with the vertex set $\{u_1,\ldots,u_r\}$ and let
$n_1,\ldots,n_{r}$ be positive integers. An $n$-vertex \emph{complete blow-up} $H=F(n_1,\ldots,n_{r})$ of $F$
is defined as follows: $ V(H)=\bigcup_{i=1}^{r}V_i$ and $ E(H)=\bigcup_{u_iu_j\in E(F)}\{uv:u\in V_i,\ v\in V_j\}$,
where $V_1,\ldots,V_{r}$ (called the \emph{parts} of $H$) are pairwise disjoint sets with $|V_i|=n_i$ and $\sum_{i=1}^r n_i=n$. For $i\in \{1,\ldots, r\}$, the value $\frac{n_i}{n}$ is called the \emph{normalized part size}. 
For each $u_i\in V(F)$, let $\boldsymbol{x}:=(x_{u_i})_{u_i\in V(F)}$ be a vector with $x_{u_i}=\frac{n_i}{n}$. 
Note that $x_{u_i}>0$ and $\sum_{i=1}^r x_{u_i}=1.$ Let $A$ be the adjacency matrix of $F$. For each vertex $u_i\in V(F)$, the \emph{normalized part degree $d_{u_i}$} of $u_i$ in $H$ is defined as follows:
$$
d_{u_i}:=\sum_{u_j\in N_F(u_i)}x_{u_j}=(A\boldsymbol{x})_i.
$$
Equivalently, every vertex in $V_i$ has degree $d_{u_i}n$ in $H$. Moreover, for each $u_i\in V(F)$, let
$\boldsymbol{b}_i:=\frac{\boldsymbol{1}_{V_i}}{\sqrt{x_{u_i} n}}.$
The vectors $\{\boldsymbol{b}_i:u_i\in V(F)\}$ form an orthonormal basis
of the subspace of vectors that are constant on every part. With respect to this basis, the restriction of $\frac{1}{n}L(H)$ is the symmetric matrix $Q:=(Q_{uv})_{u,v\in V(F)}$, where
$$Q_{uv}
    =
    \begin{cases}
        d_u             & v=u,\\
        -\sqrt{x_ux_v}, & uv\in E(F),\\
        0,              & uv\notin E(F).
    \end{cases}
$$
This matrix $Q$ is called the \emph{normalized quotient Laplacian matrix} of $H$. We now give a formal definition equivalent to the description.

\begin{definition}
For a connected graph $F$ and its $n$-vertex complete blow-up $H$, the \emph{normalized quotient Laplacian matrix} $Q$ of $H$, is defined as follows:
$$
Q=\operatorname{diag}(d_u:u\in V(F))-D_{\sqrt{x}}AD_{\sqrt{x}},\text{~~where~~} D_{\sqrt{x}}:=\operatorname{diag}(\sqrt{x_u}:u\in V(F)).
$$
The \emph{eigenvalues of $Q$} are denoted by $\theta_1(Q), \theta_2(Q),\ldots, \theta_{|V(F)|}(Q)$ with
$\theta_1(Q)\leq\theta_2(Q)\le\cdots\le\theta_{|V(F)|}(Q).$
\end{definition}

Indeed, for every vector $\boldsymbol{y}=(y_u)_{u\in V(F)}$, we have
$\boldsymbol{y}^{T}Q\boldsymbol{y}
=\sum_{uv\in E(F)}(\sqrt{x_v}y_u-\sqrt{x_u}y_v)^2\geq 0.$
Thus, $Q$ is positive semidefinite. Moreover, the equality holds if and only if $\frac{y_u}{\sqrt{x_u}}=\frac{y_v}{\sqrt{x_v}}$ for every edge $uv\in E(F)$. Since $F$ is connected, this ratio must be constant on $V(F)$. Consequently,
$\ker Q=\operatorname{span}\{
(\sqrt{x_u})_{u\in V(F)}
\},$ and thus $\dim\ker Q=1$. It follows that $0$ is an eigenvalue of $Q$ (of multiplicity $1$), and, therefore,
$0=\theta_1(Q)< \theta_2(Q)\leq\cdots\leq\theta_{|V(F)|}(Q).$ We next establish the following relation between $\mu_2(H)$ and $\theta_2(Q)$.

\begin{lemma}\label{lem:blowup-reduction-rk}
Let $H$ be an $n$-vertex complete blow-up of a connected graph $F$ with normalized part sizes $(x_u)_{u\in V(F)}$. Let $d_u$ denote the normalized part degree of $u$ such that $d_u=\sum_{v\in N_F(u)}x_v$. Then 
$$\operatorname{Spec}\big(\frac1nL(H)\big)
=\operatorname{Spec}(Q)\uplus \big(\biguplus_{u\in V(F)}\{d_u\}^{x_un-1}\big).$$
Consequently,
$\frac{\mu_2(H)}{n}=\min\big\{\theta_2(Q), \min\limits_{\substack{u\in V(F)\\ x_un
\geq 2}}d_u\big\}.$ In particular, $\frac{\mu_2(H)}{n}\le \theta_2(Q)$.
\end{lemma}

\begin{proof}
For each $u\in V(F)$, let $V_u$ denote the part of $H$ corresponding to $u$, let $\boldsymbol{f}_u:=\frac{\one_{V_u}}{\sqrt{|V_u|}}$ (where $|V_u|=x_un$), and let $U$ be the subspace spanned by
the vectors $\boldsymbol{f}_u$. As shown above, the subspace $U$ is $L(H)$-invariant, and with respect to the orthonormal basis $(\boldsymbol{f}_u)_{u\in V(F)}$, the restriction of $\frac{1}{n}L(H)$ to $U$ is represented by $Q$.

For each $u\in V(F)$, let
$W_u=\{\boldsymbol{x}\in\mathbb R^{V(H)}:
\operatorname{supp}(\boldsymbol{x})\subseteq V_u,\ 
\sum_{v\in V_u}x_v=0\}$. Then we have the following orthogonal decomposition of $\mathbb R^{V(H)}$:
$$\mathbb R^{V(H)}=U\oplus(\bigoplus_{u\in V(F)}W_u).$$ We claim that for each $u\in V(F)$, $W_u$ is an eigenspace of $\frac{1}{n}L(H)$ corresponding to the eigenvalue $d_u$, with dimension $|V_u|-1$. Let $\boldsymbol{x}$ be an arbitrary vector of $W_u$. For every $z\in V_u$, since $V_u$ is an independent vertex set and $\boldsymbol{x}$ vanishes outside $V_u$, $(L(H)\boldsymbol{x})_z=nd_u x_z$. For every $z\in V_w$ with $w\neq u$, we have that $(L(H)\boldsymbol{x})_z=0$. Indeed, if $uw\notin E(F)$, then it is trivially true; if $uw\in E(F)$, then $(L(H)\boldsymbol{x})_z=-\sum_{v\in V_u}x_v=0$. Hence, $L(H)\boldsymbol{x}=nd_u\boldsymbol{x}$. We prove the claim.

Together with the eigenvalues of $\frac{1}{n}L(H)$ on $U$, which are precisely the eigenvalues of $Q$, this yields the stated
spectrum decomposition and thus the estimate for $\mu_2(H)$.
\end{proof}

\subsection{Courant--Fischer variational principle for complex vectors}

In later proofs, we shall need the following extension of the Courant-Fischer variational principle (Theorem~2.4.1~\cite{BH2012}) to complex vectors.

\begin{lemma}\label{lem:complex-rayleigh}
Let $G$ be an $n$-vertex graph and let
$\boldsymbol{y}\in\mathbb C^n\setminus\{\mathbf 0\}$ be a complex vector.  If
$\langle\boldsymbol{y}^*, \one\rangle=0$, then
$$
 \mu_2(G)\le \frac{\boldsymbol{y}^*L(G)\boldsymbol{y}}{\langle\boldsymbol{y}^*,\boldsymbol{y}\rangle},
$$
where $\boldsymbol{y}^*$ denotes the conjugate transpose of $\mathbf{y}$. 
\end{lemma}

\begin{proof}
The Courant--Fischer variational principle \cite{BH2012} gives
$
 \mu_2(G)
 =
 \min\limits_{\substack{
          \boldsymbol{x}\in\mathbb R^n\setminus\{\mathbf0\}\\
        \langle\boldsymbol{x}^{\mathsf T},\one\rangle=0}}
 \frac{\boldsymbol{x}^{\mathsf T}L(G)\boldsymbol{x}}
      {\langle\boldsymbol{x}^{\mathsf T},\boldsymbol{x}\rangle}.
$
The condition $\langle\boldsymbol{y}^*, \one\rangle=0$ implies that both $\operatorname{Re}\boldsymbol{y}$ and $\operatorname{Im}\boldsymbol{y}$
are orthogonal to $\one$.  Applying the displayed variational
formula to each nonzero one (with the zero case being trivial) and
adding the resulting inequalities implies
{\small $$
 \boldsymbol{y}^*L(G)\boldsymbol{y}
 =
 (\operatorname{Re}\boldsymbol{y})^{\mathsf T}
 L(G)(\operatorname{Re}\boldsymbol{y})
 +
 (\operatorname{Im}\boldsymbol{y})^{\mathsf T}
 L(G)(\operatorname{Im}\boldsymbol{y})\ge
 \Big(
   |\operatorname{Re}\boldsymbol{y}|_2^2
   +
   |\operatorname{Im}\boldsymbol{y}|_2^2
 \Big)\mu_2(G)
 =\langle\boldsymbol{y}^*,\boldsymbol{y}\rangle\mu_2(G).
$$}%
Dividing by $\langle\boldsymbol{y}^*,\boldsymbol{y}\rangle$ (which is positive) proves the assertion.
\end{proof}

Remark that the orthogonality hypothesis is essential. For example, if $G$ is connected
and $\boldsymbol{y}=\one$, then $\frac{\boldsymbol{y}^*L(G)\boldsymbol{y}}{\langle\boldsymbol{y}^*,\boldsymbol{y}\rangle}=0<\mu_2(G)$.

\begin{lemma}\label{lem:weighted-variational}
Let $H$ be an $n$-vertex complete blow-up of a connected graph $F$ with normalized part sizes $(x_u)_{u\in V(F)}$. For every non-constant
function $g:V(F)\to\mathbb C$, 
$$\frac{\mu_2(H)}{n}\leq
   \frac{\sum\limits_{uv\in E(F)}
      x_ux_v \big|g(u)-g(v)\big|^2
   }{\sum\limits_{u\in V(F)}x_u \big|g(u)-\sum\limits_{u\in V(F)}x_u g(u)\big|^2
   }.
$$
\end{lemma}

\begin{proof}
For each vertex $v\in V(H)$, let $f_v:=g(u)-\sum_{u\in V(F)}x_u g(u)$ whenever $v\in V_u$ and let $\boldsymbol{f}:=(f_v)_{v\in V(H)}\in\mathbb C^{V(H)}$. Since $\sum_{u\in V(F)}x_u=1$, we have that 
\begin{align*}
\langle\boldsymbol{f}, \one\rangle
&=\sum_{v\in V(H)}f_v
=\sum_{v\in V(H)} \bigl(g(u)-\sum_{u\in V(F)}x_u g(u)\bigr)=n\sum_{u\in V(F)}x_u
\bigl(
g(u)-\sum_{u\in V(F)}x_ug(u)
\bigr)\\
&=
n\bigl(
\sum_{u\in V(F)}x_ug(u)-\bigl(\sum_{u\in V(F)}x_u\bigr)\bigl(\sum_{u\in V(F)}x_ug(u)\bigr)
\bigr)=0,
\end{align*}
and thus the vector $\boldsymbol{f}$ is orthogonal to $\one$. Since each pair of adjacent parts is complete,
$$\boldsymbol{f}^*L(H)\boldsymbol{f}=n^2\sum_{uv\in E(F)}x_ux_v \big|g(u)-g(v)\big|^2,
\text{~~and~~}
\langle\boldsymbol{f}^*,\boldsymbol{f}\rangle=n\sum_{u\in V(F)} x_u \big|g(u)-\sum_{u\in V(F)}x_u g(u)\big|^2.
$$
We conclude the proof by~\Cref{lem:complex-rayleigh}.
\end{proof}

\subsection{Induced matching conditions for the algebraic connectivity}

For certain graph structures, we may define a suitable function $g$ to establish the desired upper bound. An \emph{induced matching} of a graph $G$ is a matching $M$ such that the subgraph of $G$ induced by the endpoints of the edges in $M$ consists precisely of the edges in $M$.

\begin{lemma}\label{lem:matching}
Let $r$ be a positive real number. Let $G$ be a connected graph and let $H$ be its $n$-vertex complete blow-up. For each vertex $u\in V(G)$, let $x_{u}$ denote the normalized part size of $u$ in $H$ and let $d_{u}$ denote the normalized part degree of $u$ in $H$. If there are two edges $u_1u_2$ and $u_3u_4$ that form an induced matching in $G$ such that $\sum_{j=1}^4\frac{d_{u_j}-r}{x_{u_j}}<4$, then $\mu_2(H)<rn$. 
\end{lemma}

\begin{proof}
Assume that $G$ contains an induced matching $u_1u_2$ and $u_3u_4$ such that $\sum_{j=1}^4\frac{d_{u_j}-r}{x_{u_j}}<4$. We define $g:V(G)\to\mathbb R$ by
$$
   g(u_1)=\frac1{x_{u_1}},
   \qquad
   g(u_2)=\frac1{x_{u_2}},
   \qquad
   g(u_3)=-\frac1{x_{u_3}},
   \qquad
   g(u_4)=-\frac1{x_{u_4}},
$$
and $g(u)=0$ for any $u\in V(G)\setminus \{u_1,u_2,u_3,u_4\}$. It is easy to see that $\sum_{u\in V(G)}x_u g(u)=0$ and $\sum_{u\in V(G)}x_u|g(u)|^2=\sum_{j=1}^4\frac1{x_{u_j}}.$
Moreover,
$\sum_{uv\in E(G)}
 x_ux_v|g(u)-g(v)|^2
=\sum_{u\in V(G)} x_ud_ug(u)^2
-2\sum_{uv\in E(G)}x_ux_vg(u)g(v)=
(\sum_{j=1}^4\frac{d_{u_j}}{x_{u_j}})-4.$
Hence, $$
\sum_{uv\in E(G)}
 x_ux_v|g(u)-g(v)|^2
-r\sum_{u\in V(G)}x_u|g(u)|^2\\
=
(\sum_{j=1}^4
\frac{d_{u_j}-r}{x_{u_j}})-4
<0.$$
It then follows directly from~\Cref{lem:weighted-variational}.
\end{proof}

\section[triangle-free]{Triangle-free graphs with $\mu_2(G)\geq \frac{n}{3}$}

In this section, we shall prove the following main result.

\medskip
\noindent
{\bf \Cref*{thm:main}.} \emph{Every $n$-vertex triangle-free graph $G$ with $\mu_2(G)\geq \frac{n}{3}$ is bipartite.
Moreover, the constant $\frac{1}{3}$ is asymptotically best possible.}
\medskip

The following structural characterization theorem~\cite{BT2005} is a key ingredient in our proof; See also Theorem~5.1 of~\cite{LPR2021}.

\begin{theorem}[Brandt and Thomass\'e~\cite{BT2005}]
\label{thm:structure}
Every $n$-vertex maximal triangle-free graph $G$ with $\delta(G)>\frac{n}{3}$ is a complete blow-up of an Andr\'asfai graph or a Vega graph.
\end{theorem}

\subsection{Complete blow-ups of an Andr\'asfai graph}

We begin by establishing an upper bound on the algebraic connectivity of complete blow-ups of Andr\'asfai graphs.

\begin{proposition}\label{prop:andrasfai-blowup}
For any integer $k$ with $k\geq 2$, every $n$-vertex complete blow-up $H$ of an Andr\'asfai graph $\Gamma_k$ satisfies $\mu_2(H)<\frac{n}{3}$.
\end{proposition}

We consider two cases: $k=2$ and $k\geq 3$. 

\subsubsection[C5]{Complete blow-ups of $C_5$}

Note that the Andr\'asfai graph $\Gamma_2$ is isomorphic to the odd cycle $C_5$. Assume that $H$ is an $n$-vertex complete blow-up of $C_5$ with normalized part sizes being $(a_i)_{i\in \mathbb{Z}_5}$ (where all indices are taken modulo $5$). Note that $a_i>0$ for each $i\in \mathbb{Z}_5$ and $\sum_{i=1}^5 a_i=1$. By the definition, the normalized quotient Laplacian matrix $Q=(Q_{ij})_{1\le i,j\le 5}$ of $H$ is as follows:
$$
Q_{ij}=
\begin{cases}
a_{i-1}+a_{i+1}, & j=i,\\
-\sqrt{a_i a_j}, & j=i-1 \text{ or } j=i+1,\\
0, & \text{otherwise}.
\end{cases}
$$
Let $\beta_1, \beta_2, \ldots, \beta_5$ denote the eigenvalues of $Q$ such that $\beta_1<\beta_2\le\beta_3\le\beta_4\le\beta_5$, where $\beta_1=0$. For values $x_1,\ldots,x_m$, let
$$
\psi_1(x_1,\ldots,x_m):=\sum_{1\le i\le m}x_i,~
\psi_2(x_1,\ldots,x_m):=\sum_{1\le i<j\le m}x_ix_j,
~\text{and}~
\psi_3(x_1,\ldots,x_m):=\sum_{1\le i<j<k\le m}x_ix_jx_k.
$$

\begin{lemma}\label{lem:traceQ}
We have that
\begin{itemize}
\setlength{\itemsep}{0em}
    \item $\psi_1(\beta_2,\beta_3,\beta_4,\beta_5)=2$.
    \item $\psi_2(\beta_2,\beta_3,\beta_4,\beta_5)=\frac{1}{2}(3-\sum_{i=1}^5a_i^2)$.
    \item $\psi_3(\beta_2,\beta_3,\beta_4,\beta_5)=\frac{1}{3}-\frac{1}{3}\sum_{i=1}^5a_i^3+2\sum_{i=1}^5a_ia_{i+1}a_{i+2}$.
\end{itemize}
\end{lemma}

\begin{proof}
Note that $\beta_1=0$. Obviously $\psi_1(\beta_2,\beta_3,\beta_4,\beta_5)= \operatorname{tr}Q=\sum_{i=1}^5(a_{i-1}+a_{i+1})=2$. For the second claim,  
$\operatorname{tr}Q^2=\sum_{i=1}^5(a_{i-1}+a_{i+1})^2+2\sum_{i=1}^5 a_ia_{i+1}=1+\sum_{i=1}^5 a_i^2$, where the last equality is because of the fact $(\sum_{i=1}^5 a_i)^2=1$. Then $\psi_2(\beta_2,\beta_3,\beta_4,\beta_5)=\frac{(\operatorname{tr}Q)^2-\operatorname{tr}Q^2}{2}=\frac{1}{2}(3-\sum_{i=1}^5a_i^2)$.

For the third one, note that $Q=D-B$, where $D=\diag(a_{i-1}+a_{i+1})$ and $B$ has nonzero entries $B_{i,i+1}=B_{i+1,i}=\sqrt{a_ia_{i+1}}$, and all other entries of $B$ are zero. We have that 
{\small
\begin{align*}
\operatorname{tr}Q^3&=\operatorname{tr}\bigl(
D^3-D^2B-DBD-BD^2+DB^2+BDB+B^2D-B^3
\bigr)\\
&=\operatorname{tr}D^3+3\operatorname{tr}(DB^2)-\operatorname{tr}B^3\\
&=\sum_{i=1}^5 (a_{i-1}+a_{i+1})^3+3\sum_{i=1}^5 a_i(a_{i-1}+a_{i+1})^2\\
&=2\sum_{i=1}^5a_i^3+3\sum_{i=1}^5a_i^2(a_{i-2}+a_{i+2})+3\sum_{i=1}^5a_i^2(a_{i-1}+a_{i+1})+6\sum_{i=1}^5a_ia_{i+1}a_{i+2}\\
&=2\sum_{i=1}^5a_i^3+3\sum_{i=1}^5a_i^2(1-a_i)+6\sum_{i=1}^5a_ia_{i+1}a_{i+2}\\
&=3\sum_{i=1}^5a_i^2-\sum_{i=1}^5a_i^3+6\sum_{i=1}^5a_ia_{i+1}a_{i+2}.
\end{align*}}%
Indeed, since $D$ is diagonal and $B$ has zero diagonal, $\operatorname{tr}(D^2B)=\operatorname{tr}(DBD)=\operatorname{tr}(BD^2)=0,$
while, by the cyclicity of the trace,
$\operatorname{tr}(DB^2)
=\operatorname{tr}(BDB)
=\operatorname{tr}(B^2D)
$; Since $C_5$ is triangle-free $\operatorname{tr}B^3=0$, and, moreover, $(B^2)_{ii}=a_i(a_{i-1}+a_{i+1})$. 

Hence,
$\psi_3(\beta_2,\beta_3,\beta_4,\beta_5)
=\frac{1}{6}\big(
(\operatorname{tr}Q)^3
-3\operatorname{tr}Q\operatorname{tr}Q^2
+2\operatorname{tr}Q^3\big)
=\frac{1}{3}-\frac{1}{3}\sum_{i=1}^5a_i^3
+2\sum_{i=1}^5 a_ia_{i+1}a_{i+2}.$
This completes the proof.
\end{proof}

For each $j\in \{2,3,4,5\}$, let $\gamma_j:=\beta_j-\frac{1}{3}.$ Note that
{\small 
\begin{align*}
\psi_3(\gamma_2,\gamma_3,\gamma_4,\gamma_5)
&=\sum_{2\leq i<j<k\leq 5}
 (\beta_i-\frac{1}{3})
 (\beta_j-\frac{1}{3})
 (\beta_k-\frac{1}{3})\\
&=\psi_3(\beta_2,\beta_3,\beta_4,\beta_5)
-\frac{1}{3}\sum_{2\leq i<j<k\leq 5}(\beta_i\beta_j+\beta_i\beta_k+\beta_j\beta_k)+\frac{1}{9}\sum_{2\leq i<j<k\leq 5}
 (\beta_i+\beta_j+\beta_k)
-\frac{4}{27}\\
&=
\psi_3(\beta_2,\beta_3,\beta_4,\beta_5)-\frac{2}{3}\psi_2(\beta_2,\beta_3,\beta_4,\beta_5)
+\frac{1}{3}\psi_1(\beta_2,\beta_3,\beta_4,\beta_5)
-\frac{4}{27}\\
&=
\frac{1}{3}-\frac{1}{3}\sum_{i=1}^5a_i^3
+2\sum_{i=1}^5a_ia_{i+1}a_{i+2}-\frac{1}{3}\big(3-\sum_{i=1}^5a_i^2\big)
+\frac{2}{3}-\frac{4}{27}\\
&=
\frac{1}{3}\big(
\sum_{i=1}^5(a_i^2-a_i^3)
+6\sum_{i=1}^5a_ia_{i+1}a_{i+2}
-\frac{4}{9}
\big),
\end{align*}}%
where the third equality follows from the fact that each product $\beta_i\beta_j$ occurs in exactly two of the four triples, while each $\beta_i$ occurs in exactly three of them, and the fourth equality follows from~\Cref{lem:traceQ}.

\begin{lemma}\label{lem:C5-cubic}
If $a_i\geq 0$ for each $i\in \Z_5$ and $\sum_{i=1}^5 a_i=1$, then $\psi_3(\gamma_2,\gamma_3,\gamma_4,\gamma_5)\leq 0$. 
Moreover, the equality holds only if $(a_1,a_2,a_3,a_4,a_5)=(\frac{1}{3},\frac{1}{3},\frac{1}{3},0,0)$ 
up to cyclic permutation.
\end{lemma}

\begin{proof}
Let $F:=\sum_{i=1}^5 (a_i^2-a_i^3)+6\sum_{i=1}^5 a_ia_{i+1}a_{i+2}$. The calculation above shows that it suffices to prove that $F\leq \frac{4}{9}$ for $(a_1,\ldots,a_5)$ satisfying that $a_i\geq 0$ for every $i\in \Z_5$ and $\sum_{i=1}^5a_i=1$.

We first claim that $F$ attains its maximum at some $(a_1,\ldots,a_5)$ for which at least one coordinate $a_i$ is zero. Assume to the contrary that $F$ attains its maximum at $\boldsymbol{a}=(a_1,\ldots,a_5)$ with $a_i>0$ for every $i\in \Z_5$. Let $\varepsilon$ be a real number such that $0<\varepsilon<\min_i a_i$. For each $i\in\Z_5$, let $\boldsymbol a_i^+
:=(a_i+\varepsilon,a_{i+1}+\varepsilon,a_{i+2},a_{i+3}-\varepsilon,a_{i+4}-\varepsilon)$ and $\boldsymbol a_i^-:=
(a_i-\varepsilon,a_{i+1}-\varepsilon,a_{i+2},a_{i+3}+\varepsilon,a_{i+4}+\varepsilon)$. Since $F(\boldsymbol{a})$ is the maximum, we have that for each $i$, $F(\boldsymbol {a}_i^{+})+F(\boldsymbol{a}_i^{-})-2F(\boldsymbol a)\leq 0$. However, since $\sum_i a_i=1$, a direct computation gives that $\sum_{i=1}^5 \big(F(\boldsymbol a_i^{+})+F(\boldsymbol a_i^{-})-2F(\boldsymbol a)\big)=\sum_{i=1}^5 \big(2\varepsilon^2(
4-3a_i-9a_{i+1}+6a_{i+2}-9a_{i+3}-3a_{i+4})\big)=4\varepsilon^2>0$, a contradiction. Hence, since $F$ is invariant under cyclic permutations of $(a_1,\ldots,a_5)$, by symmetry we may assume that
$a_5=0$. Then we have that $$F=(a_1+a_4)(a_2+a_3)+\bigl(3(a_1+a_4)-2\bigr)a_1a_4+\bigl(1+3(a_1+a_4)\bigr)a_2a_3.$$ We consider the following two cases.

If $0\leq a_1+a_4\leq \frac{2}{3}$, then $3(a_1+a_4)-2\leq 0$. Since $a_1a_4\geq0$ and $a_2a_3\leq \frac{(a_2+a_3)^2}{4}=\frac{(1-a_1-a_4)^2}{4}$, we obtain that
$F\leq(a_1+a_4)(1-a_1-a_4)+\frac{1+3(a_1+a_4)}{4}(1-a_1-a_4)^2.$
Consequently, $$F-\frac{4}{9}\leq\frac{\bigl(3(a_1+a_4)-7\bigr)\bigl(3(a_1+a_4)-1\bigr)^2}{36}\leq 0,$$
where the equality holds if and only if $a_1+a_4=\frac{1}{3}$, $a_1a_4=0$, and $a_2=a_3$; that is, by the symmetry of $a_1$ and $a_4$, $a_1=a_2=a_3=\frac{1}{3}$ and $a_4=0$.

If $\frac{2}{3}\leq a_1+a_4\leq 1$, then both $3(a_1+a_4)-2$ and $1+3(a_1+a_4)$ are nonnegative. Since $a_1a_4\leq\frac{(a_1+a_4)^2}{4}$ and $a_2a_3\leq\frac{(1-a_1-a_4)^2}{4}$, we have $$F\leq\frac{6(a_1+a_4)^3-11(a_1+a_4)^2+5(a_1+a_4)+1}{4}=:\varphi(a_1+a_4).$$
We then consider the function $\varphi(x)=\frac{6x^3-11x^2+5x+1}{4}$ for $x\in[\frac{2}{3},1]$. By calculation, we know $\varphi'(x)=\frac{18x^2-22x+5}{4}$. Then $\varphi'(x)<0$ for $x\in [\frac{2}{3},\frac{11+\sqrt{31}}{18})$ and $\varphi'(x)>0$ for $x\in (\frac{11+\sqrt{31}}{18},1]$. Thus, $\varphi(x)$ decreases on $[\frac{2}{3},\frac{11+\sqrt{31}}{18}]$ and increases on $[\frac{11+\sqrt{31}}{18},1]$. Therefore, the maximum of $\varphi(x)$ occurs at one endpoint of the interval $[\frac{2}{3},1]$. Since $\varphi(\frac{2}{3})=\frac{11}{36}<\frac{4}{9},$ and $\varphi(1)=\frac{1}{4}<\frac{4}{9},$ no equality occurs in this range. The equality conditions in the
first range give precisely the cyclic permutations of
$(\frac{1}{3},\frac{1}{3},\frac{1}{3},0,0)$.
\end{proof}

By \Cref{lem:C5-cubic}, $\psi_3(\gamma_2,\gamma_3,\gamma_4,\gamma_5)\leq 0.$ Since $H$ is a complete blow-up of $C_5$, we have $a_i>0$ for every
$i\in\Z_5$ and thus the equality case in \Cref{lem:C5-cubic} cannot occur. Therefore,
$\psi_3(\gamma_2,\gamma_3,\gamma_4,\gamma_5)<0$. If $\beta_2\ge \frac{1}{3}$, then, since $\beta_2\le\beta_3\le\beta_4\le\beta_5$, we have $\gamma_j\geq 0$ for all $2\leq j\leq 5$, and thus $\psi_3(\gamma_2,\gamma_3,\gamma_4,\gamma_5)\ge0$, a contradiction. So $\beta_2<\frac{1}{3}$. It then follows from~\Cref{lem:blowup-reduction-rk} that $\frac{\mu_2(H)}{n}\le\beta_2<\frac{1}{3}$.

\subsubsection[Gammak]{Complete blow-ups of Andr\'asfai graphs $\Gamma_k$ for $k\geq 3$}
Let $k\geq 3$. Let $V(\Gamma_k)=\{v_i\mid i\in \Z_{3k-1}\}$ and thus $N_{\Gamma_k}(v_i)=\{v_{i+k},v_{i+k+1},\ldots,v_{i+2k-1}\}.$ Let $V_i$ be the part of $H$ corresponding to the vertex $v_i\in V(\Gamma_k)$, and let $x_i:=\frac{|V_i|}{n}$ and $ d_i:=\sum_{t=k}^{2k-1}x_{i+t}$. 
Note that $x_i>0$, $\sum_{i\in \Z_{3k-1}} x_i=1$, and every vertex in $V_i$ has degree $nd_i$ in $H$. If $d_i<\frac{1}{3}$ for some $i$, then, since $H$ is non-complete, \Cref{lem:Fiedler} implies $\mu_2(H)\leq\delta(H)\leq nd_i<\frac{n}{3}$.
Thus we may assume that for each $i\in\Z_{3k-1}$, $d_i\geq \frac{1}{3}$. By the definition,
$d_{i-k}+d_i+d_{i+k}=\sum_{t=0}^{k-1}x_{i+t}+\sum_{t=k}^{2k-1}x_{i+t}+\sum_{t=2k}^{3k-1}x_{i+t}=\sum_{t=0}^{3k-1}x_{i+t}=1+x_i,$ and equivalently, 
\begin{equation}\label{equ:x_i}
(d_{i-k}-\frac{1}{3})+(d_i-\frac{1}{3})+(d_{i+k}-\frac{1}{3})=x_i. 
\end{equation}
For each $i\in \Z_{3k-1}$, we have that $0\leq \frac{d_i-\frac{1}{3}}{x_i}\leq 1$. Summing \eqref{equ:x_i} over all $i\in\Z_{3k-1}$ gives $1=\sum_{i\in\Z_{3k-1}} x_i=3\sum_{i\in\Z_{3k-1}}(d_i-\frac{1}{3}).$
Thus there exists an index $s$ such that $d_s>\frac{1}{3}$. Let $s':=s+k$.
Since $d_{s'-k}=d_s>\frac{1}{3}$, \eqref{equ:x_i} at this index shows that $\frac{d_{s'}-\frac{1}{3}}{x_{s'}}<1$.

Moreover, the two edges $v_{s'}v_{s'+k+1}$ and $v_{s'+2}v_{s'+2k}$ form an induced matching in $\Gamma_k$. Indeed, the four cross-differences of indices are $2$, $2k$, $k-1$, and $k-1$; since $k\geq 3$, none of which belongs to the set $\{k,k+1,\ldots,2k-1\}$. Noting that $ \frac{d_{s'}-\frac{1}{3}}{x_{s'}}+\frac{d_{s'+k+1}-\frac{1}{3}}{x_{s'+k+1}}+\frac{d_{s'+2}-\frac{1}{3}}{x_{s'+2}}+\frac{d_{s'+2k}-\frac{1}{3}}{x_{s'+2k}}<4,$ \Cref{lem:matching}, applied with $G=\Gamma_k$ and $r=\frac{1}{3}$, implies that $\mu_2(H)<\frac{n}{3}$.

\subsection{Complete blow-ups of a Vega graph}

We next consider the algebraic connectivity of complete blow-ups of Vega graphs.

\begin{proposition}\label{prop:vega-blowup}
Every $n$-vertex complete blow-up $H$ of a Vega graph satisfies
$\mu_2(H)<\frac{n}{3}$.
\end{proposition}

\begin{proof}
Let $F=\Upsilon_i^{\mu\nu}$ be a Vega graph and let $H$ be an $n$-vertex complete blow-up of $F$. Let $V_z$ be the part of
$H$ corresponding to $z\in V(F)$, and let
$x_z:=\frac{|V_z|}{n}$ and
$d_z:=\sum_{w\in N_F(z)}x_w$. Note that $x_z>0$,
$\sum_{z\in V(F)}x_z=1$, and every vertex in $V_z$ has degree
$nd_z$ in $H$. If $d_z<\frac{1}{3}$ for some $z\in V(F)$, then, since
$H$ is non-complete, \Cref{lem:Fiedler} implies
$\mu_2(H)\leq\delta(H)\leq nd_z<\frac n3$. Thus we may assume that
$d_z\geq\frac{1}{3}$ for every $z\in V(F)$. In what follows, we adopt the convention that any term indexed by a vertex deleted in the construction of $\Upsilon_i^{\mu\nu}$ is omitted. More precisely, terms indexed by $y$ are omitted when $\mu=1$, whereas those indexed by $v_{2i-1}$ are omitted when $\nu=1$. Thus, the corresponding terms are simply absent from the relevant expressions.

Assume first that $(i,\mu,\nu)\neq (2,1,1)$. Let
$$
S_0:=\{v_j\in V(F):j\ne2i-1\}\text{~and~}
S_1:=
\begin{cases}
\{v_j:j\ne0\},&\text{~if~}\nu=0,\\
\{v_j:j\ne i-1,2i-1\},&\text{~if~} \nu=1.
\end{cases}
$$
Note that the edges $av$ and $ub$ form an induced matching in $F$, since they are opposite edges of the induced $6$-cycle $avcubw$. By the definition of the Vega graph (\Cref{def:vega}) and the fact that $\sum_{z\in V(F)}x_z=1$, we have that
{\footnotesize
$$
\begin{aligned}
x_a={}&\sum_{t\in S_0}(d_t-\frac{1}{3})
+(d_x-\frac{1}{3})+i(d_a-\frac{1}{3})
+(i-1)\sum_{t\in\{b,c,u,v,w\}}(d_t-\frac{1}{3}),\\
x_b={}&\sum_{t\in S_1}(d_t-\frac{1}{3})
+(d_x-\frac{1}{3})
+(i-1)\sum_{t\in\{a,c,u,w\}}(d_t-\frac{1}{3})
+(i-\nu)(d_b-\frac{1}{3})
+(i-1-\nu)(d_v-\frac{1}{3}),\\
x_u={}&\sum_{t\in S_0}(d_t-\frac{1}{3})
+\mu(d_x-\frac{1}{3})
+(1-\mu)(d_y-\frac{1}{3})
+(i-1)\sum_{t\in\{a,b,c\}}(d_t-\frac{1}{3})\\
&+(i-\mu)(d_u-\frac{1}{3})
+(i-1-\mu)\sum_{t\in\{v,w\}}(d_t-\frac{1}{3}),\\
x_v={}&\sum_{t\in S_1}(d_t-\frac{1}{3})
+\mu(d_x-\frac{1}{3})
+(1-\mu)(d_y-\frac{1}{3})
+(i-1)\sum_{t\in\{a,c\}}(d_t-\frac{1}{3})\\
&+(i-1-\nu)(d_b-\frac{1}{3})
+(i-1-\mu)\sum_{t\in\{u,w\}}(d_t-\frac{1}{3})
+(i-\mu-\nu)(d_v-\frac{1}{3}).
\end{aligned}
$$
}
Since, for every $z\in V(F)$, $d_z\geq \frac{1}{3}$ and all the coefficients above are nonnegative, we have that
$$
\frac{d_a-\frac{1}{3}}{x_a}\leq\frac1i,\qquad
\frac{d_b-\frac{1}{3}}{x_b}\leq\frac1{i-\nu},\qquad
\frac{d_u-\frac{1}{3}}{x_u}\leq\frac1{i-\mu},\text{~~and~~}
\frac{d_v-\frac{1}{3}}{x_v}\leq\frac1{i-\mu-\nu}.
$$
If $i\geq3$, the sum of these four quantities is at most
$\frac{1}{3}+\frac{1}{2}+\frac{1}{2}+1<4$, while if $i=2$ and
$(\mu,\nu)\ne (1,1)$, it is at most $3$. Thus
$\frac{d_a-\frac{1}{3}}{x_a}+\frac{d_v-\frac{1}{3}}{x_v}
+\frac{d_u-\frac{1}{3}}{x_u}+\frac{d_b-\frac{1}{3}}{x_b}<4$.
By~\Cref{lem:matching}, applied with $G=F$ and $r=\frac{1}{3}$, we obtain that
$\mu_2(H)<\frac n3$.

\medskip

It remains to consider the case when $F=\Upsilon_2^{11}$, for which we use the labeling of $\Upsilon_2^{11}$ shown on the right side of~\Cref{fig:Vega211}.
Note that $p_0p_1p_2p_3p_4p_0$ is an induced $5$-cycle of $F$, $N_F(q_j)=\{p_{j-1},p_{j+1},s\}$, and
$N_F(s)=\{q_0,\ldots,q_4\}$, where all the indices are taken modulo $5$. A direct calculation gives
{\small
$$
\begin{aligned}
x_{p_j}=&
\sum_{t=0}^4 (d_{p_t}-\frac{1}{3})
+(d_s-\frac{1}{3})
+\sum_{t=j-1}^{j+1} (d_{q_t}-\frac{1}{3}),\\
x_{q_j}=&
\sum_{t=j-1}^{j+1} (d_{p_t}-\frac{1}{3})
+(d_{q_{j+2}}-\frac{1}{3})
+(d_{q_{j+3}}-\frac{1}{3})
+(d_s-\frac{1}{3}),\\
x_s=&
\sum_{t=0}^4\big(
(d_{p_t}-\frac{1}{3})
+(d_{q_t}-\frac{1}{3})\big)
+2(d_s-\frac{1}{3}).
\end{aligned}
$$}%
Thus $\frac{d_{p_j}-\frac{1}{3}}{x_{p_j}}\leq1$ and
$\frac{d_s-\frac{1}{3}}{x_s}\leq\frac{1}{2}$ for every $j\in\Z_5$.

We claim that $\frac{d_{q_j}-\frac{1}{3}}{x_{q_j}}<1$ for some
$j\in\Z_5$. Otherwise, $d_{q_j}-\frac{1}{3}\geq x_{q_j}$ for
every $j$, and thus
$d_{q_j}-\frac{1}{3}\geq
\sum_{t=j-1}^{j+1}(d_{p_t}-\frac{1}{3})
+(d_{q_{j+2}}-\frac{1}{3})
+(d_{q_{j+3}}-\frac{1}{3})
+(d_s-\frac{1}{3}).$
Summing over $j\in\Z_5$ gives
$\sum_{j=0}^4(d_{q_j}-\frac{1}{3})
\geq
\sum_{j=0}^4\sum_{t=j-1}^{j+1}
(d_{p_t}-\frac{1}{3})
+\sum_{j=0}^4(d_{q_{j+2}}-\frac{1}{3}) 
+\sum_{j=0}^4(d_{q_{j+3}}-\frac{1}{3})
+5(d_s-\frac{1}{3})=
3\sum_{j=0}^4(d_{p_j}-\frac{1}{3})
+2\sum_{j=0}^4(d_{q_j}-\frac{1}{3})
+5(d_s-\frac{1}{3})$, and
equivalently,
$$
3\sum_{j=0}^4 (d_{p_j}-\frac{1}{3})
+\sum_{j=0}^4 (d_{q_j}-\frac{1}{3})
+5(d_s-\frac{1}{3})\leq 0.
$$
Since $d_z\ge \frac{1}{3}$ for every $z\in V(F)$, all the terms
$d_{p_j}-\frac{1}{3}$, $d_{q_j}-\frac{1}{3}$, and $d_s-\frac{1}{3}$
are nonnegative, and thus $d_{p_j}=d_{q_j}=d_s=\frac{1}{3}$. Hence, we have that
$x_{p_0}=0$, contradicting the fact $x_{p_0}>0$. Therefore, there exists an index $j^*\in\Z_5$ such that $\frac{d_{q_{j^*}}-\frac{1}{3}}{x_{q_{j^*}}}<1$.

Let $t\in\Z_5$ such that $t:=j^*-3$. Now we consider the edges
$p_tp_{t+1}$ and $q_{j^*}s$ that form an induced matching of $F$. Indeed, $N_F(q_{j^*})=\{p_{t+2},p_{t+4},s\}$, and $s$ is adjacent to none of the vertices $p_i$. Therefore
$\frac{d_{p_t}-\frac{1}{3}}{x_{p_t}}
+\frac{d_{p_{t+1}}-\frac{1}{3}}{x_{p_{t+1}}}
+\frac{d_{q_{j^*}}-\frac{1}{3}}{x_{q_{j^*}}}
+\frac{d_s-\frac{1}{3}}{x_s}
<1+1+1+\frac{1}{2}<4$.
By \Cref{lem:matching}, applied with $G=F$ and $r=\frac{1}{3}$, we conclude
that $\mu_2(H)<\frac n3$.
\end{proof}

\subsection[mainThm]{Proof of \Cref*{thm:main}}

\begin{proof}[Proof of \Cref*{thm:main}]
Assume to the contrary that $G$ is an $n$-vertex triangle-free graph with $\mu_2(G)\geq \frac{n}{3}$ that is not bipartite. Thus $G$ is
connected and not a complete graph. By~\Cref{lem:Fiedler}, $\delta(G)\ge\mu_2(G)\ge \frac{n}{3}$. 

We first claim that $\delta(G)>\frac{n}{3}$.
Assume to the contrary that $\delta(G)=\mu_2(G)=\frac{n}{3}$. Let $S$ be a minimum vertex cut of $G$.
Since $G$ is connected and not complete, $|S|\le\delta(G)=\frac{n}{3}$.
Let $A$ be the vertex set of one component of $G-S$, and let
$B:=V(G)\setminus(A\cup S)$. We define a real vector $\boldsymbol{f}=(f_v)_{v\in V(G)}$ by
$$
f_v=
   \begin{cases}
      \frac{1}{|A|}, &v\in A,\\
      -\frac{1}{|B|}, &v\in B,\\
      0, &v\in S.
   \end{cases}
$$
Noting that $\langle\boldsymbol{f},\one\rangle=\sum_{v\in 
A} \frac{1}{|A|}-\sum_{v\in 
B} \frac{1}{|B|}=0$, the vector $\boldsymbol{f}$ is orthogonal to $\one$. By~\Cref{lem:complex-rayleigh}, 
we have that 
$$
   \frac{n}{3}=\mu_2(G)
   \le
   \frac{\boldsymbol{f}^{T}L(G)\boldsymbol{f}}
        {\langle \boldsymbol{f}^{T},\boldsymbol{f}\rangle}
   =
   \frac{\frac{e(A,S)}{|A|^2}+\frac{e(B,S)}{|B|^2}}{\frac{1}{|A|}+\frac{1}{|B|}}\leq |S|\le\delta(G)=\frac{n}{3},
$$
where the second equality is because there is no edge between $A$ and $B$, and the second inequality follows from the fact that $e(A,S)\le |A||S|, e(B,S)\le |B||S|$.
Thus every equality above should hold. In particular, $e(A,S)=|A||S|$ and $e(B,S)=|B||S|$, that is to say, every vertex of $S$ is adjacent to every vertex of $V(G)\setminus S$. Since $G$ is triangle-free, both $S$ and $V(G)\setminus S$ are independent. Hence, $G$ is a complete bipartite graph, a contradiction. Therefore, $\delta(G)>\frac{n}{3}$.

We now add possible edges to $G$, preserving triangle-freeness, until obtaining a maximal triangle-free graph $\widetilde{G}$.
By~\Cref{lem:edge-monotone}, $\mu_2(G)\le\mu_2(\widetilde G)$, and
clearly $\delta(\widetilde G)>\frac{n}{3}$. By~\Cref{thm:structure}, and since $G$ is not bipartite, $\widetilde G$ is a complete blow-up of a non-bipartite Andr\'asfai graph or a Vega graph. In both cases, it follows from~\Cref{prop:andrasfai-blowup} and \Cref{prop:vega-blowup} that $\mu_2(G)\le\mu_2(\widetilde G)<\frac{n}{3}$, a contradiction. 

\medskip
We next show that the constant $\frac{1}{3}$ is asymptotically best possible.
For $\ell\geq 1$, let $B_\ell$ be a complete blow-up of the $5$-cycle defined by $B_\ell:=C_5(\ell,1,1,\ell,\ell),$
where the part sizes are listed in the cyclic order. Clearly, $B_\ell$ is
triangle-free and non-bipartite, and $|V(B_\ell)|=3\ell+2$. By~\Cref{prop:andrasfai-blowup} with $k=2$, $\frac{\mu_2(B_\ell)}{|V(B_\ell)|}<\frac{1}{3}$. Moreover, the normalized part sizes of
$B_\ell=C_5(\ell,1,1,\ell,\ell)$ are $\Big(
\frac{\ell}{3\ell+2},\frac{1}{3\ell+2},\frac{1}{3\ell+2},\frac{\ell}{3\ell+2},\frac{\ell}{3\ell+2}\Big)$,
which converge to $\Big(\frac{1}{3},0,0,\frac{1}{3},\frac{1}{3}\Big)$
as $\ell\to\infty$. Hence, by the definition of the normalized quotient matrix $Q_\ell$, for sufficiently large $\ell$, $Q_\ell$ converges to
{\small
$$Q_\infty=
   \begin{pmatrix}
      \frac{1}{3}&0&0&0&-\frac{1}{3}\\
      0&\frac{1}{3}&0&0&0\\
      0&0&\frac{1}{3}&0&0\\
      0&0&0&\frac{1}{3}&-\frac{1}{3}\\
      -\frac{1}{3}&0&0&-\frac{1}{3}&\frac{2}{3}
   \end{pmatrix}.
$$}%
A direct calculation gives
$\operatorname{Spec}(Q_\infty)=\{0,\frac{1}{3},\frac{1}{3},\frac{1}{3},1\}$. Since the
eigenvalues of symmetric matrices depend continuously on their
entries, $\theta_2(Q_\ell)\to\frac{1}{3}$. Noting that the only parts of
$B_\ell=C_5(\ell,1,1,\ell,\ell)$ having size greater than one are
the first, fourth, and fifth parts (when $\ell\geq 2$), by \Cref{lem:blowup-reduction-rk}, we have that
$\frac{\mu_2(B_\ell)}{3\ell+2}=\min\{\theta_2(Q_\ell),\frac{\ell+1}{3\ell+2},\frac{2\ell}{3\ell+2}\}$.
Therefore, as $\ell\to\infty$, 
$\lim_{\ell\to\infty}\frac{\mu_2(B_\ell)}{|V(B_\ell)|}=\frac{1}{3}.
$ This proves that the constant $\frac{1}{3}$ in \Cref{thm:main} is
asymptotically best possible.
\end{proof}

\section[large-girth]{Graphs of odd girth at least $2k+1$ with $\mu_2(G)>\frac{4n}{6k-1}$}

In this section, we shall prove the following main result.

\medskip
\noindent
{\bf \Cref*{thm:main-girth}.} \emph{Let $k$ be an integer with $k\geq 3$. Every $n$-vertex graph $G$ of odd girth at least $2k+1$ with $\mu_2(G)>\frac{4n}{6k-1}$ is bipartite.}
\medskip

Our proof of~\Cref{thm:main-girth} is based on the following structural result.

\begin{theorem}[Lu and Wang~\cite{LW2026+}]
\label{thm:structure2}
For every integer $k\geq 2$, every $n$-vertex graph of odd girth at least $2k+1$ with $\delta(G)>\frac{4n}{6k-1}$ admits a homomorphism to $K_{4k/(2k-1)}$.
\end{theorem}

\subsection[Circular clique]{Complete blow-ups of a circular clique $K_{4k/(2k-1)}$ or $K_{(2k+1)/k}$}

We consider the algebraic connectivity of complete blow-ups of a circular clique $K_{4k/(2k-1)}$ or $K_{(2k+1)/k}$. Note that $K_{(2k+1)/k}$ is isomorphic to the odd cycle $C_{2k+1}$ for each positive integer $k$.

\begin{proposition}\label{prop:uniform K_{4k/2k-1}andK_{2k+1/k}}
For any integer $k$ with $k\geq 2$, every $n$-vertex complete blow-up $H$ of a circular clique $K_{(2k+1)/k}$ or $K_{4k/(2k-1)}$ satisfies $\mu_2(H)<\frac{4n}{6k-1}$.
\end{proposition}

\begin{proof}
Let $H$ be an $n$-vertex complete blow-up of $G$ where $G\in \{C_{2k+1}, K_{4k/(2k-1)}\}$. Let $s$ be defined by $s=1$ if $G=C_{2k+1}$ and $s=2$ if $G=K_{4k/(2k-1)}$. Let $m=2k+2-s$. It is easy to see that $G$ has $sm$ vertices. We use the labeling
$V(G)=\{v_j:j\in\Z_{sm}\}$ in which $N_G(v_j)=\{v_{j-1},v_{j+1}\}$ if $s=1$, and $N_G(v_j)=\{v_{j+m-1},v_{j+m},v_{j+m+1}\}$ if $s=2$. 
Let $V_j$ be the part of $H$ corresponding to $v_j$, and define its
normalized size $x_j:=\frac{|V_j|}{n}$ and normalized part degree
$$d_j:=
   \begin{cases}
      x_{j-1}+x_{j+1},&s=1,\\
      x_{j+m-1}+x_{j+m}+x_{j+m+1},&s=2.
   \end{cases}
$$
All subscripts on $V_j,x_j,$ and $d_j$ are taken modulo $sm$.

Assume to the contrary that
$\mu_2(H)\geq \frac{4n}{6k-1}$. The graph $H$ is connected and
non-complete, so by~\Cref{lem:Fiedler}, for each $j\in\Z_{sm}$, $d_j\geq\frac{4}{6k-1}$.
Moreover, $\sum_{j=1}^{sm} x_j=1$ and $ \sum_{j=1}^{sm} d_j=s+1$. 

For $i\in\Z_m$, we define that $ a_i:=\sum_{t=0}^{s-1}x_{i+tm}$, that is to say,
$$
  a_i=
   \begin{cases}
      x_i,&\text{~if~}s=1,\\
      x_i+x_{i+2k},&\text{~if~}s=2.
   \end{cases}
$$ In both cases, $\sum_{i=1}^m a_i=1$.
And, with subscripts on $a_i$ taken modulo $m$,
\begin{equation}\label{eq:unified-block-degrees}
   \sum_{t=0}^{s-1}d_{i+tm}
   =a_{i-1}+(s-1)a_i+a_{i+1}
   \geq\frac{4s}{6k-1}.
\end{equation}
Let $\mathrm{i}=\sqrt{-1}$, and let
$$
   \omega:=e^{\frac{2\pi\mathrm{i}}{m}}
   \text{~~and~~} 
   z:=\sum_{j\in\Z_m}a_j\omega^j
   \text{~where~} \omega^j:=e^{\frac{2\pi\mathrm{i}}{m}j}.
$$ Note that $\omega^{m}=1$ and $\sum_{j\in\Z_{m}}\omega^j=\frac{1-\omega^{m}}{1-\omega}=0$.
We define an $n$-dimensional complex vector
$\boldsymbol{f}=(f_u)_{u\in V(H)}$ by setting
$$
   f_u:=\omega^j-z, 
   \text{~~for~~} u\in\bigcup_{t=0}^{s-1}V_{j+tm}.
$$
Furthermore, let $\mathcal W$ denote the
normalized total weight of the links joining consecutive blocks:
$\mathcal W:=\sum_{j\in\Z_m}\sum_{t=0}^{s-1}x_{j+tm}x_{j+1+(s-1-t)m}, 
$ that is to say, 
$$
  \mathcal W=
   \begin{cases}
      \sum\limits_{j\in \Z_{2k+1}} x_jx_{j+1},&\text{~if~}s=1,\\
      \sum\limits_{j\in\Z_{4k}}x_jx_{j+2k-1},&\text{~if~}s=2.
   \end{cases}
$$

Since every $a_j$ is positive and the numbers $\omega^j$ are not
all equal, $z$ is a strict convex combination of distinct points on
the unit circle. Hence, $|z|<1$. By the definitions of $a_j$, $z$,
and $\boldsymbol{f}$, we have that $\langle\boldsymbol{f}^*,\mathbf1\rangle
   =n\sum_i a_i(\overline{\omega^i}-\overline z)=0$ and $\langle\boldsymbol{f}^*,\boldsymbol{f}\rangle
   =n(1-|z|^2)>0$.
Thus we obtain that
$$\boldsymbol{f}^*L(H)\boldsymbol{f}=\sum_{u,v\in V(H)}\overline{f_u}L(H)_{uv}f_v=\sum_{uv\in E(H)}|f_u-f_v|^2=4n^2\mathcal W \sin^2(\frac{\pi}{m}),$$ where every unordered edge is counted once in the second last sum.  
\Cref{lem:weighted-variational} now yields
\begin{equation}\label{eq:unified-rayleigh-bound}
   \frac{\mu_2(H)}n
   \leq\frac{4\sin^2(\frac{\pi}{m})}{1-|z|^2}\mathcal W.
\end{equation}

\begin{claim}\label{lem:common-sin-bound}
We have that $\sin^2(\frac{\pi}{m})<\frac{30}{3m^2+10}$. 
\end{claim}

\begin{proof*}[Proof of \Cref*{lem:common-sin-bound}]
We aim to show that $\sin^2t<\frac{t^2}{1+\frac{t^2}{3}}$. Let
$F(t):=\frac{t^2}{\sin^2 t}-1-\frac{t^2}{3}$.
Since $\lim_{t\to 0}F(t)=0$, it suffices to prove that $F'(t)>0$ for $0<t<\pi$. We have that $F'(t)=\frac{2t}{\sin^2 t}-\frac{2t^2\cos t}{\sin^3 t}-\frac{2t}{3}=\frac{2t}{3\sin^3 t}\bigl(3(\sin t-t\cos t)-\sin^3 t\bigr)$.
Let $G(t):=3(\sin t-t\cos t)-\sin^3 t$. Then
$G'(t)=3\sin t\bigl(t-\sin t\cos t\bigr)$. Now let $H(t):=t-\sin t\cos t$. Since $H(0)=0$ and $H'(t)=1-\cos(2t)=2\sin^2t>0$ for $0<t<\pi$, we have $H(t)>0$. As $\sin t>0$, it follows that $G'(t)>0$. Together with $G(0)=0$, this gives $G(t)>0$, and thus $F'(t)>0$.

Therefore, $F(t)>0$, or equivalently,
$\sin^2 t<\frac{t^2}{1+\frac{t^2}{3}}$. Since $\pi^2<10$, taking $t=\frac{\pi}{m}$ we obtain 
$\sin^2(\frac{\pi}{m})<\frac{3\pi^2}{3m^2+\pi^2}<\frac{30}{3m^2+10},$
as claimed.
\end{proof*}

\begin{claim}\label{lem:unified-z-bound}
We have that $|z|<\frac{3-s}{6}$. 
\end{claim}

\begin{proof*}[Proof of \Cref*{lem:unified-z-bound}]
Since $\sum_{j\in\Z_{m}}\omega^j=0$, 
\eqref{eq:unified-block-degrees} and the triangle inequality give
{\small 
$$
\big(s-1+2\cos(\frac{2\pi}{m})\big)|z|=
   \big|\sum_{j\in \Z_m}
   \big(a_{j-1}+(s-1)a_j+a_{j+1}
   -\frac{4s}{6k-1}\big)\omega^j\big|\leq s+1-\frac{4sm}{6k-1}=\frac{(3-s)(2k-3)}{6k-1},
$$}%
where the last equality uses the fact that $s\in\{1,2\}$.  The coefficient of
$|z|$ is strictly larger than $\frac{6(2k-3)}{6k-1}$. Indeed, \Cref{lem:common-sin-bound} implies
$s-1+2\cos(\frac{2\pi}{m})=s+1-4\sin^2(\frac{\pi}{m})>s+1-\frac{120}{3m^2+10},
$
and thus
$$
s+1-\frac{120}{3m^2+10}-\frac{6(2k-3)}{6k-1}
=
\begin{cases}
\frac{8\big(5+(k-2)(24k-18)\big)}
 {(6k-1)\big(3(2k+1)^2+10\big)},&\text{~if~}s=1,\\
\frac{6\big(41+(k-2)(12k^2+54k-2)\big)}
 {(6k-1)(12k^2+10)},&\text{~if~}s=2.
\end{cases}
$$
Both quantities are positive for $k\geq 2$, proving the claim.
\end{proof*}

\begin{claim}\label{lem:unified-weight}
We have that $$\mathcal W\leq
   \begin{cases}
      \frac{(2k+1)^2}{(6k-1)^2},&\text{~if~}s=1,\\
      \frac{13k-11}{8(6k-1)},&\text{~if~}s=2.
   \end{cases}$$ 
\end{claim}

\begin{proof*}[Proof of \Cref*{lem:unified-weight}]
Suppose first that $s=1$. Iterating
$x_j=d_{j-1}-x_{j-2}$ around the odd cycle gives $2x_j=\sum_{t=0}^{2k}(-1)^t d_{j-(2t+1)}$.
Since multiplication by $2$ permutes $\Z_{2k+1}$, the
indices in this sum run through all vertices exactly once.
Since for each $j\in \Z_{2k+1}$, $d_j\geq\frac{4}{6k-1}$ and $\sum_{j\in \Z_{2k+1}} d_j=2$, we obtain that
$$
2x_j=\frac{4}{6k-1}+\sum_{t=0}^{2k}(-1)^t(d_{j-(2t+1)}-\frac{4}{6k-1})\leq\frac{4}{6k-1}+\sum_{t=1}^{2k+1}(d_t-\frac{4}{6k-1})=2-\frac{8k}{6k-1}.
$$
Hence, $x_j\leq \frac{2k-1}{6k-1}$.  Therefore,
{\footnotesize $$
2\mathcal{W}
=2\sum\limits_{j\in \Z_{2k+1}} x_jx_{j+1}
=\sum_{j\in \Z_{2k+1}} x_jd_j=\frac{4}{6k-1}+\sum_j x_j(d_j-\frac{4}{6k-1})
\leq\frac{4}{6k-1}+\frac{2k-1}{6k-1}(2-\frac{4(2k+1)}{6k-1})
=\frac{2(2k+1)^2}{(6k-1)^2}.
$$}

Now suppose that $s=2$. Since $d_j\geq\frac{4}{6k-1}$ for each $j\in \Z_{4k}$ and 
$d_{t+2k}=x_{t-1}+x_t+x_{t+1}$, applying it with $t=j-1,j+1,j+4$ gives
$$
   \frac{12}{6k-1}
   \leq d_{j+2k-1}+d_{j+2k+1}+d_{j+2k+4}
   =x_j+\sum_{h=-2}^{5}x_{j+h}
   \leq1+x_j.
$$
Hence, $x_j\geq \frac{13-6k}{6k-1}$. 
For $k=2$, this gives
$x_j\geq \frac{1}{11}$ and thus $ \sum_{j\in\Z_4}x_jx_{j+4}=\sum_{j\in\Z_4}(x_j-\frac{1}{11})(x_{j+4}-\frac{1}{11})+\frac{1}{11}-\frac{4}{121}\geq\frac{7}{121}>\frac{5}{88}$.
For $k\geq 3$, it follows that $\sum_{i\in\Z_{2k}}x_ix_{i+2k}\geq\frac{5(3-k)}{8(6k-1)}$, noting that the quantity
$\frac{5(3-k)}{8(6k-1)}$ is nonpositive.
Furthermore, since $d_j\geq\frac{4}{6k-1}$ for each $j\in \Z_{4k}$ and $\sum_{j=1}^{4k} d_j=3$, we have that $ d_j\leq3-\frac{4(4k-1)}{6k-1}
   =\frac{2k+1}{6k-1}.$
The antipodal and non-antipodal links partition the edge set, so
$$
\begin{aligned}
   \mathcal W+\sum_{i\in\Z_{2k}}x_ix_{i+2k}=\frac12\sum_{j\in\Z_{4k}}x_jd_j\leq\frac{2k+1}{2(6k-1)}.
\end{aligned}
$$
The asserted bound on $\mathcal W$ follows.
\end{proof*}

Finally, combining \eqref{eq:unified-rayleigh-bound} with \Cref{lem:common-sin-bound}, \Cref{lem:unified-z-bound}, and \Cref{lem:unified-weight} gives
$$
   \frac{\mu_2(H)}n
   <
   \begin{cases}
      \frac{135(2k+1)^2}
      {(6k-1)^2(3(2k+1)^2+10)},&\text{~if~}s=1,\\
      \frac{108(13k-11)}
      {7(6k-1)(12k^2+10)},&\text{~if~}s=2.
   \end{cases}
$$
Both bounds are smaller than $\frac{4}{6k-1}$.  Indeed, $4(6k-1)\bigl(3(2k+1)^2+10\bigr)-135(2k+1)^2=365+(k-2)(288k^2+276k+276)>0,$ whereas $ 84k^2-351k+367=1+(k-2)(84k-183)>0.$
Thus $\frac{\mu_2(H)}{n}< \frac{4}{6k-1}$, contradicting the assumption $\mu_2(H)\geq \frac{4n}{6k-1}$. This completes the proof.
\end{proof}

\subsection[mainThm2]{Proof of \Cref*{thm:main-girth}}

\begin{proof}[Proof of \Cref*{thm:main-girth}]
Let $k$ be an integer with $k\geq 3$. Assume to the contrary that $G$ is an $n$-vertex graph of odd girth at least $2k+1$ with $\mu_2(G)>\frac{4n}{6k-1}$, which is not bipartite. Since $\mu_2(G)>\frac{4n}{6k-1}>0$, $G$ is connected. Moreover, since $G$ has odd girth at least $2k+1\geq 7$, $G$ is triangle-free and not a complete graph. By
\Cref{lem:Fiedler}, $\delta(G)\ge\mu_2(G)>\frac{4n}{6k-1}$.
Hence, by~\Cref{thm:structure2}, there is a homomorphism $\varphi$ of $G$ to $K_{4k/(2k-1)}$.

We claim that $\varphi$ is a surjective homomorphism of $G$ to a graph $F$ such that $F$ is isomorphic to either $K_{4k/(2k-1)}$ or $C_{2k+1}$. Assume not and by~\Cref{lem:K_4k/2k-1ToC_2k+1}, the homomorphic image $\varphi(G)$ is a proper subgraph of $C_{2k+1}$ (which is homomorphic to $K_2$), contradicting the fact that $G$ is not bipartite. 

In both cases, the pre-image of each vertex of $F$ is not empty. Now we add all possible edges connecting $V_i$ and $V_j$ whenever $ij\in E(F)$ and the resulting graph $\widetilde{G}$ is a complete blow-up of $F$. For each part, the corresponding normalized part degree in $\widetilde{G}$ is at least $\frac{\delta(G)}{n}$, and hence is strictly larger than $\frac{4}{6k-1}$. Since $G\subseteq \widetilde{G}$, \Cref{lem:edge-monotone} and \Cref{prop:uniform K_{4k/2k-1}andK_{2k+1/k}} together imply $\mu_2(G)\le\mu_2(\widetilde{G})<\frac{4n}{6k-1}$, a contradiction. This completes the proof of the theorem.
\end{proof}

\section[Asymptotic bound]{An asymptotic bound when $k$ is large}

We prove the following result without attempting to optimize the absolute constant $3456$.

\medskip
\noindent
{\bf \Cref*{thm:main-girth_asymptotical}} {\emph{For any integer $k$ with $k\geq 22$, every $n$-vertex graph $G$ of odd girth at least $2k+1$ with $\mu_2(G)>\frac{3456n}{k^3}$ is bipartite. Moreover, the order $k^{-3}$ is asymptotically best possible.}}
\medskip

Note that for any large integer $n$ and any integer $s$ such that $\frac{n}{3}<s<\frac{n}{2}$, by~\Cref{lem:precise-alg-connectivity}\ref{lem:Kab} $\mu_2(K_{s,n-s})=s>\frac{n}{3}$. Thus, for any positive integer $k\geq 22$, $\mu_2(K_{s,n-s})>\frac{n}{3}>\frac{3456n}{k^3}$. That is to say, there exists a family of $n$-vertex graphs $G$ such that $\mu_2(G)>\frac{3456n}{k^3}$.

\begin{lemma}\label{lem:coarse-mass-resistance}
Let $m$ be a positive integer and let $x_0,\ldots,x_m$ be positive real numbers. There exists an index $j\in\{0,\ldots,m-1\}$ such that
$$
\big(\sum_{i=0}^m x_i\big)
\big(\sum_{i=0}^j x_i\big)
\big(\sum_{i=j}^{m-1}\frac{1}{x_i x_{i+1}}\big)
\geq
\frac{m^3}{216}.
$$
\end{lemma}

\begin{proof}
We proceed by induction on $m$.
The case $m=1$ is immediate, since letting $j=0$ we have that $(\sum_{i=0}^1 x_i)x_0\frac{1}{x_0x_1}=\frac{x_0+x_1}{x_1}>1>\frac{1}{216}$.

Assume now $m\geq 2$. Let $q:=\left\lfloor\frac m2\right\rfloor,$ and $r:=m-q.$ Note that $\frac{m}{3}\leq q<m$ and $\frac{m}{2}\leq r< m$.

We first consider the case when $\sum_{i=0}^{q}x_i\leq\big(\frac qm\big)^3\sum_{i=0}^{m}x_i$. Applying the induction hypothesis to the sequence $x_0,\ldots,x_q$, there exists $j<q$ such that
$(\sum_{i=0}^{q}x_i)
(\sum_{i=0}^j x_i)
(\sum_{i=j}^{q-1}\frac{1}{x_ix_{i+1}})
\geq\frac{q^3}{216}.$
Since $\sum_{i=0}^{m}x_i\geq \frac{m^3}{q^3}\sum_{i=0}^{q}x_i$, we obtain that
$$
\big(\sum_{i=0}^m x_i \big)
\big(\sum_{i=0}^j x_i\big)
\big(\sum_{i=j}^{m-1}\frac{1}{x_ix_{i+1}}\big)
\geq \frac{m^3}{q^3}\big(\sum_{i=0}^q x_i \big)\big(\sum_{i=0}^j x_i\big)\big(\sum_{i=j}^{m-1}\frac{1}{x_ix_{i+1}}\big)
\geq \frac{m^3}{216}.
$$
where the last inequality follows from the fact that every term $\frac{1}{x_ix_{i+1}}$ is positive and $m>q$. 

It remains to consider the case when $\sum_{i=0}^{q}x_i> \big(\frac qm\big)^3\sum_{i=0}^{m}x_i$. In this case, we shall prove that $(\sum_{i=0}^m x_i)
(\sum_{i=0}^q x_i)
(\sum_{i=q}^{m-1}\frac{1}{x_i x_{i+1}}) \geq\frac{m^3}{216}$, that is the case of $j=q$. On the one hand, by H\"older's inequality,
$
\big(\sum_{i=q}^{m-1}\frac{1}{x_ix_{i+1}}\big)
\big(\sum_{i=q}^{m-1}\sqrt{x_ix_{i+1}}\big)^2
\geq r^3.
$
On the other hand, the arithmetic--geometric mean inequality gives
$
\sum_{i=q}^{m-1}\sqrt{x_ix_{i+1}}
\leq
\frac{1}{2}\sum_{i=q}^{m-1}(x_i+x_{i+1})
\leq
\sum_{i=q}^{m}x_i
\leq\sum_{i=0}^{m}x_i.
$
Therefore,
$(\sum_{i=0}^{m}x_i)^2 (\sum_{i=q}^{m-1}\frac{1}{x_ix_{i+1}})
\geq r^3.$
Consequently,
{\small 
$$
\big(\sum_{i=0}^{m}x_i\big)\big(\sum_{i=0}^{q}x_i\big)\big(\sum_{i=q}^{m-1}\frac{1}{x_ix_{i+1}}\big)
\geq \frac{\sum_{i=0}^{q}x_i}{\sum_{i=0}^{m}x_i}r^3
>
m^3 \big(\frac qm\big)^3 \big(\frac{r}{m}\big)^3\
\geq m^3\big(\frac{1}{3}\big)^3\big(\frac{1}{2}\big)^3
=\frac{m^3}{216}.
$$}%
It completes the proof of the lemma.
\end{proof}

We now use this lemma to obtain an upper bound of the algebraic connectivity of a graph in terms of its diameter. For a connected graph $G$, the \emph{diameter} of $G$ is the maximum distance between any two vertices of $G$.

\begin{theorem}\label{thm:diameter-bound}
Every $n$-vertex connected graph with diameter $D\geq 1$ satisfies $\mu_2(G)\leq \frac{3456 n}{D^3}$.
\end{theorem}

\begin{proof}
Let $u,v\in V(G)$ be two vertices such that $d_G(u,v)=D$. For each $0\leq i\leq D$, let
$L_i:=\{x\in V(G):d_G(u,x)=i\}$, called the $i$th layer of $G$ (with respect to $u$). Note that each $L_i$ is nonempty. Moreover, every edge of $G$ has both endpoints in
the same layer or in two consecutive layers, because adjacent vertices have
distances from $u$ differing by at most one.

Let $H$ be obtained from $G$ by adding edges such that every $G[L_i]$ forms a clique and every pair $(L_i,L_{i+1})$ is completely bipartite. Then $G\subseteq H$. By~\Cref{lem:edge-monotone}, $\mu_2(G)\leq\mu_2(H)$. Moreover, note that $d_H(u,v)=D$, since every edge of $H$ joins vertices in the same layer or in consecutive layers. So we now aim to bound $\mu_2(H)$.

For each $0\leq i\leq D$, let $a_i:=\frac{|L_i|}{n}$. Then $a_i>0$ and $\sum_{i=0}^D a_i=1$. Let $p$ be the smallest index such that $\sum_{i=0}^{p}a_i\geq\frac{1}{2}$. Note that at least one of $p$ and $D-p$ is at least $\frac{D}{2}$. We have that  
\begin{equation}\label{eq:median-mass}
\frac{1}{2}\leq \sum_{i=0}^{p}a_i\leq 1, ~~~  \frac{1}{2}\leq\sum_{i=p}^{D}a_i\leq 1~~~ 
\frac{\sum_{i=0}^{p}a_i}{\sum_{i=p}^{D}a_i}\leq 2, \text{~~~and~~~}
\frac{\sum_{i=p}^{D}a_i}{\sum_{i=0}^{p}a_i}\leq 2.
\end{equation}
We first assume $p\geq\frac{D}{2}.$ Applying \Cref{lem:coarse-mass-resistance} to the sequence $a_0,\ldots,a_p$, there exists $j<p$ such that 
\begin{equation}\label{eq:AR-lower}
\Big(\sum_{i=0}^{p}a_i\Big)
\Big(\sum_{i=0}^j a_i\Big) \Big(\sum_{r=j}^{p-1}\frac{1}{a_ra_{r+1}}\Big)\geq\frac{p^3}{216}.
\end{equation}

We define a function $f:V(H)\to\mathbb R$ by assigning the value $f_i$ to every vertex in $L_i$, where $f_i=1$ for $0\leq i\leq j$, $f_i=0$ for $p\leq i\leq D$, and $f_i=\frac{\sum_{r=i}^{p-1}\frac{1}{a_ra_{r+1}}}{\sum_{r=j}^{p-1}\frac{1}{a_ra_{r+1}}}$ for $j<i<p$. Let $f^*:=\frac{1}{n}\sum_{u\in V(H)}f(u)=\sum_{i=0}^{D}a_i f_i.$ Since $f$ is constant on each layer $L_i$, the edges inside one layer contribute nothing to the Laplacian quadratic form. Moreover, there are
$|L_i||L_{i+1}|=n^2a_i a_{i+1}$ edges between $L_i$ and $L_{i+1}$. Viewing $H$ as a complete blow-up of itself (that is, for every $v\in V(H)$, the normalized part size $x_v$ satisfies $x_v=\frac{1}{n}$), we apply \Cref{lem:weighted-variational} to the non-constant function $f(v)$ and obtain that
\begin{equation}\label{eq:weighted-path-Rayleigh}
\frac{\mu_2(H)}{n}
\leq
\frac{\sum\limits_{i=0}^{D-1}a_i a_{i+1}(f_{i+1}-f_i)^2}{\sum\limits_{i=0}^{D}a_i(f_i-f^*)^2}.
\end{equation}
Note that for $j\leq i<p$, $f_i-f_{i+1}=\frac{1}{a_ia_{i+1}\sum\limits_{i=j}^{p-1}\frac{1}{a_ra_{r+1}}},$ and for $0\le i<j$ and $p \le i<D$, $f_i-f_{i+1}=0$. Hence,
\begin{equation}\label{eq:test-energy}
\sum_{i=0}^{D-1}a_ia_{i+1}(f_{i+1}-f_i)^2
=
\frac{1}{\sum\limits_{i=j}^{p-1}\frac{1}{a_ia_{i+1}}}.
\end{equation}
We next consider the denominator of the fraction of~\eqref{eq:weighted-path-Rayleigh}. Since $\sum_{i=0}^{D}a_i=1$ and
$f^*=\sum_{i=0}^{D}a_if_i$, we have
\begin{equation}\label{eq:test-variance}
\begin{aligned}
\sum_{i=0}^{D}a_i(f_i-f^*)^2
&=\sum_{i=0}^{D}a_if_i^2-(f^*)^2
=\frac{1}{2}\sum_{i=0}^{D}\sum_{\ell=0}^{D} a_i a_\ell(f_i-f_\ell)^2
=\sum_{0\le i<\ell\le D}
a_i a_\ell(f_i-f_\ell)^2\\
&\geq
\sum_{\substack{0\le i\le j\\ p\le \ell\le D}}
a_i a_\ell(f_i-f_\ell)^2=
\sum_{\substack{0\le i\le j\\ p\le \ell\le D}}
a_i a_\ell=
\big(\sum_{i=0}^{j}a_i\big)
\big(\sum_{\ell=p}^{D}a_\ell\big).  
\end{aligned}
\end{equation}
Indeed, the second equality follows from
$\frac{1}{2}\sum_{i,\ell}a_i a_\ell(f_i-f_\ell)^2=\frac{1}{2}\sum_{i,\ell}a_i a_\ell
(f_i^2+f_\ell^2-2f_if_\ell)
=\sum_i a_i f_i^2-(\sum_i a_i f_i)^2
=\sum_i a_i f_i^2-(f^*)^2$, and the inequality is because for $0\leq i\leq j$ we have $f_i=1$, while for $p\leq\ell\leq D$ we
have $f_\ell=0$. 
Combining~\eqref{eq:weighted-path-Rayleigh}, \eqref{eq:test-energy}, and
\eqref{eq:test-variance}, we obtain
$$
\frac{\mu_2(H)}{n}
\leq \frac{\sum\limits_{i=0}^{D-1}a_i a_{i+1}(f_{i+1}-f_i)^2}{\sum\limits_{i=0}^{D}a_i(f_i-f^*)^2}
\leq \frac{1}{\big(\sum\limits_{i=j}^{p-1}\frac{1}{a_ia_{i+1}}\big) \big(\sum\limits_{i=0}^j a_i\big) 
\big(\sum\limits_{i=p}^{D}a_i \big)}.
$$
Applying \eqref{eq:median-mass} and \eqref{eq:AR-lower}, and since $p\geq \frac{D}{2}$,
$$
\frac{\mu_2(H)}n
\leq  \frac{216\sum\limits_{i=0}^{p}a_i}{p^3\sum\limits_{i=p}^{D}a_i}
\leq  \frac{432}{p^3}
\leq  \frac{3456}{D^3}.
$$

If instead $D-p\geq \frac{D}{2}$, we apply the same argument to the reversed
sequence $a_D,a_{D-1},\ldots,a_p$.  This gives exactly the same bound.  Thus
in all cases $\mu_2(G)\leq\mu_2(H)\leq\frac{3456n}{D^3}$.
\end{proof}

The following lemma is a folklore; for example, see Proposition~23.1 in~\cite{B1993}.

\begin{lemma}{\rm \cite{B1993}}
\label{lem:odd girth-diameter}
Let $k$ be a positive integer. Every connected graph $G$ of odd girth at least $2k+1$ has diameter at least $k$.
\end{lemma}

The diameter bound immediately yields the $k^{-3}$ order for graphs
of large odd girth.

\begin{proof}[Proof of~\Cref*{thm:main-girth_asymptotical}] 
Let $k$ be an integer with $k\geq 22$. Assume to the contrary that $G$ is an $n$-vertex graph of odd girth at least $2k+1$ with $\mu_2(G)>\frac{3456n}{k^3}$ that is not bipartite. Since $\mu_2(G)>0$, $G$ is connected. By~\Cref{lem:odd girth-diameter},
the diameter of $G$ is at least $k$. Hence, \Cref{thm:diameter-bound} implies that $\mu_2(G) \leq \frac{3456n}{k^3},$
a contradiction. 

\medskip
To show $k^{-3}$ is tight, we consider all the odd cycles $C_{2k+1}$. Note that $C_{2k+1}$ has odd girth exactly $2k+1$ and, by~\Cref{lem:precise-alg-connectivity}\ref{lem:C_2k+1}, $\mu_2(C_{2k+1})
=2-2\cos(\frac{2\pi}{2k+1})$.
Therefore,
$\frac{\mu_2(C_{2k+1})}{|V(C_{2k+1})|}=\frac{2-2\cos(\frac{2\pi}{2k+1})}{2k+1}=\frac{\pi^2}{2k^3}+O(k^{-4}).$
\end{proof}

\section{Remarks and questions}\label{sec:Remarks}

We make some remarks and pose some further questions in this section.

\begin{remark}\label{rmk:C5}
Every $n$-vertex triangle-free graph with $\mu_2(G)\geq \frac{n}{3}$ admits a homomorphism to $C_5$.
Moreover, the constant $\frac{1}{3}$ is asymptotically best possible.
\end{remark}

Indeed, for $\ell\geq 1$, let $T_\ell:=K_{8/3}(\ell,1,1,\ell,1,1,\ell,1)$, where the part sizes correspond to the vertices $0,1,\ldots,7$ of $K_{8/3}$, respectively. Clearly, $T_\ell$ is triangle-free and $|V(T_\ell)|=3\ell+5$. By~\Cref{lem:K_4k/2k-1ToC_2k+1}, $K_{8/3}$ admits no homomorphism to $C_5$. Since $T_\ell$ contains a copy of $K_{8/3}$, $T_\ell$ also admits no homomorphism to $C_5$. Moreover, the normalized part sizes of $T_\ell$ are $(\frac{\ell}{3\ell+5},\frac{1}{3\ell+5},\frac{1}{3\ell+5},\frac{\ell}{3\ell+5},\frac{1}{3\ell+5},\frac{1}{3\ell+5},\frac{\ell}{3\ell+5},\frac{1}{3\ell+5}),$ which converge to $(\frac{1}{3},0,0,\frac{1}{3},0,0,\frac{1}{3},0
)$ as $\ell\to\infty$. Hence, let $Q_\infty$ denote the limit of the normalized quotient matrix $Q_\ell$ when $\ell\to\infty$. 
A direct calculation gives $\operatorname{Spec}(Q_\infty)=\left\{0,\left(\frac{1}{3}\right)^{[6]},1\right\}$ and  $\theta_2(Q_\ell)\to\frac{1}{3}$. By~\Cref{lem:blowup-reduction-rk},
$\frac{\mu_2(T_\ell)}{3\ell+5}=
\min\{\theta_2(Q_\ell),\frac{\ell+2}{3\ell+5},\frac{2\ell+1}{3\ell+5}\}.$
As $\ell\to\infty$, $
\lim_{\ell\to\infty}\frac{\mu_2(T_\ell)}{|V(T_\ell)|}=
\frac{1}{3}.$
This implies that the constant $\frac{1}{3}$ is asymptotically best possible.

It would be interesting to determine the optimal bound under which a graph of large odd girth must be bipartite, as the bound $\frac{4}{6k-1}$ in~\Cref{thm:main-girth} does not appear to be tight.

\begin{question}\label{que:f(k)}
Determine the optimal function $f(k)$ such that every $n$-vertex graph $G$ of odd girth at least $2k+1$ satisfying $\mu_2(G)>f(k)n$ is bipartite.
\end{question}

For a graph $H$ with at least one edge, we define
$$
\tau_k(H):=\sup\{\frac{\mu_2(G)}{|V(G)|}:G \text{~has odd girth at least~} 2k+1\ \text{and}\ G\nrightarrow H\}.
$$
Thus $\tau_k(H)$ is the optimal normalized algebraic-connectivity threshold forcing a homomorphism of $G$ to $H$.

Together with its asymptotic sharpness statement,
\Cref{thm:main} implies that
$\tau_2(K_2)=\frac{1}{3}$, whereas
\Cref{thm:main-girth} yields
$\tau_k(K_2)\leq \frac{4}{6k-1}$ for every $k\geq 3$.
Moreover, \Cref{thm:main-girth_asymptotical} shows that
$\tau_k(K_2)=\Theta(k^{-3})$ as $k\to\infty$. Thus \Cref{que:f(k)} amounts to determining the exact value of $\tau_k(K_2)$ for each $k\geq 3$. Beyond bipartiteness, it is natural to study the algebraic-connectivity
homomorphism thresholds for the two circular-clique targets $C_{2k+1}$ and $K_{4k/(2k-1)}$. 
\begin{question}
Determine the exact values of $\tau_k(C_{2k+1})$ and $\tau_k(K_{4k/(2k-1)})$, and characterize the corresponding extremal graphs.
\end{question}

\Cref{rmk:C5} actually implies that $\tau_2(C_5)=\frac{1}{3}$. Since $K_2\to C_{2k+1}\to K_{4k/(2k-1)}$,  the monotonicity of the homomorphism thresholds gives $\tau_k(K_{4k/(2k-1)}) \leq \tau_k(C_{2k+1})
\leq \tau_k(K_2).$
It would be interesting to determine whether both inequalities are strict for every $k\geq 3$ and to characterize the corresponding
extremal graphs.

More generally, for $1\leq\ell\leq k$, we are interested in determining $\tau_k(C_{2\ell+1})$, namely, the optimal normalized algebraic-connectivity threshold forcing a homomorphism to $C_{2\ell+1}$ among graphs of odd girth at least $2k+1$.

\section*{Declaration of AI usage}
During the development and preparation of this manuscript, the authors made limited use of ChatGPT 5.6 Pro to explore potential approaches, assist with calculations, and improve the presentation of the manuscript. All AI-assisted suggestions and arguments were carefully reviewed, revised, and independently verified by the authors. The authors take full responsibility for the originality of the problems, the mathematical verification, the final formulation of all proofs, and the accuracy and content of the manuscript.

\section*{Acknowledgments}
The authors would like to thank Professor Hong-Jian Lai for inspiring them to investigate the connections between spectral graph parameters and graph homomorphism properties. This work is the authors' first exploration of problems at this interface.

Zhouningxin Wang is supported by National Natural Science Foundation of China (Nos. 12301444, 12671413) and the Fundamental Research Funds for the Central Universities, Nankai University.

\end{document}